\documentclass[10pt,a4paper]{article}
\usepackage[latin,english]{babel}
\usepackage[utf8]{inputenc}
\usepackage[T1]{fontenc}
\usepackage{eufrak, color, mathrsfs}
\usepackage{geometry}
\usepackage{amsmath}
\usepackage{amssymb}
\usepackage{amsthm}
\usepackage{mathtools}
\usepackage{mathabx}
\usepackage{cases}
\usepackage{braket}
\usepackage[toc,page]{appendix}
\usepackage{pdfsync}
\usepackage{hyperref}
\usepackage{psfrag}
\usepackage{bbm}
\usepackage{math}

\newtheorem{lemma}{Lemma}[section]
\newtheorem{theorem}[lemma]{Theorem}
\newtheorem{proposition}[lemma]{Proposition}

\newtheorem{definition}[lemma]{Definition}

\numberwithin{equation}{section}

\renewcommand{\bar}{\overline}

\newcommand{\lV}{\bf V}
\newcommand{\bkn}{{\bf k}}

\newcommand{\nla}[1]{{\norm{#1}_{\sL}^2}}
\newcommand{\pla}[2]{{\left(#1 , #2\right)_{\sL}}}
\newcommand{\sL}{\mathscr{L}}
\newcommand{\sub}{G}

\title{Long time dynamics of 
Schr\"odinger and wave 
equations \\
on  flat tori}

\author{
M. Berti\footnote{ International School for Advanced Studies (SISSA), Via Bonomea 265, 34136, Trieste, Italy \newline
 \textit{Email: } \texttt{berti@sissa.it}} 
,
A. Maspero\footnote{ International School for Advanced Studies (SISSA), Via Bonomea 265, 34136, Trieste, Italy \newline
 \textit{Email: } \texttt{alberto.maspero@sissa.it}}
}

\begin{document}
\maketitle

\begin{abstract}
We consider a class of linear time dependent Schr\"odinger equations and   
quasi-periodically forced nonlinear Hamiltonian 
wave/Klein Gordon and Schr\"odinger equations on arbitrary  flat tori.
 For the linear Schr\"odinger equation, we prove a $t^\epsilon$ $(\forall \epsilon >0)$ upper bound for 
 the growth of the Sobolev norms as the time goes to infinity.
 For the nonlinear Hamiltonian PDEs we construct families of time quasi-periodic solutions.
 
 Both results are based on   ``clusterization properties" 
 of the eigenvalues of the Laplacian on a 
 flat torus and 
 on suitable ``separation properties" of {the singular  sites of  Schr\"odinger
 and wave operators}, 
 {which are integers}, in space-time Fourier lattice,  close to a cone or a paraboloid. 
 Thanks to these properties we are able to apply Delort abstract theorem \cite{del2} to control the speed of growth of the Sobolev norms, and 
 Berti-Corsi-Procesi abstract Nash-Moser  theorem \cite{BCP} to construct quasi-periodic solutions.
\end{abstract}


\section{Introduction}

In the last years many efforts have been made to understand the long time dynamics of Schr\"odinger and wave equations on compact manifolds.
Two important 
problems regard  upper and lower bounds for the  possible 
growth of the Sobolev norms of a solution, and the existence of global in time quasi-periodic solutions,
for which the Sobolev norm remains perpetually bounded.
Results concerning these problems have been obtained when the underlying manifold is either 
the {standard} 
torus $ \T^d = \R^d / (2\pi\Z)^d $ (see e.g. 
\cite{bou99, bourgain99, del2} for growth of Sobolev norms for linear systems and \cite{bou98, Bou05, EK10,   geng11, bebo12, bebo13, PP15,EGK,WDuke} for quasi-periodic solutions),
a Zoll manifold (see \cite{del2, maro, BGMR2} for growth  and \cite{BBP, GP2, BK} for quasi-periodic solutions), a Lie group or a homogeneous space \cite{BP, BCP, CHP}. 
The reason is that suitable information about the eigenvalues and eigenfunctions of the Laplace operator are relevant for 
the long time dynamics. 

In this paper we extend some of these results to the case the manifold is any {\it flat} torus
\begin{equation}\label{flat:torus}
\T^d_\sL := \R^d / \sL  
\end{equation}
where $ \sL$ is {a} lattice of $ \R^d $, with 
linearly independent generators 
$\bv_1, \ldots, \bv_d \in \R^d $, i.e.  
\begin{equation}
\label{def:Gamma}
\sL := \Big\{ {\mathop \sum}_{i=1}^d m_i \bv_i \ \colon \  m_i \in \Z \Big\} \, . 
\end{equation}
A particular case of \eqref{flat:torus}-\eqref{def:Gamma}  is 
the rectangular  torus 
\begin{equation}
\label{torus}
\T^d_\tL := (\R/\tL_1 \Z)\times \cdots \times (\R/\tL_d \Z)   
\end{equation}
with generators $ \bv_a = \tL_a {\bf e}_a $   
where $\tL_a >0  $ and 
$  ({\bf e}_a)_{a=1, \ldots, d} $ denotes  the canonical basis of $ \Z^d $. 

The equations we consider are  the 
 linear time dependent  Schr\"odinger equation
\begin{equation}
\label{sch}
\im \partial_t \psi = - \Delta \psi + V(t, x) \psi  \, ,  \quad 
x \in \T^d_\sL \,  , 
\end{equation}
where $V(t,x)$ is a smooth real valued potential,  
periodic on the lattice $\sL$, namely $V(t, x + \sum_{i}m_i \bv_i) = V(t,x), $ $ \forall m_i \in \Z$, 
and the quasi-periodically forced Hamiltonian nonlinear  wave (NLW) and nonlinear Schr\"odinger (NLS)  equations
\begin{equation}
\label{NLWS}
u_{tt} - \Delta u + m u = \epsilon f(\omega t, x, u) \, , \qquad
\im u_t - \Delta u + m u = \epsilon \tf(\omega t, x, u) \, , \quad x \in \T^d_\sL \, , 
\end{equation}
where the nonlinearities $ f, \tf$ are sufficiently regular, and the frequency vector 
$ \omega \in \R^n $ is  Diophantine.

For equation \eqref{sch}, we will show that,  on {\em any flat torus} $\T^d_\sL $, 
the  Sobolev norms of any solution grow at most as  $t^\epsilon$ ($\, \forall \epsilon >0$) when $t$ goes to infinity, see Theorem \ref{thm:main}.
Concerning the nonlinear wave and Schr\"odinger  equations \eqref{NLWS}
on any {\em flat torus} $ \T^d_\sL $, we 
construct families of quasi-periodic solutions 
{for a large set of  frequency vectors $ \omega = \lambda \bar \omega  $ with a 
given Diophantine direction $ \bar \omega \in \R^n $}, see Theorem \ref{main2}. 
In particular, both results hold for rectangular tori.

\smallskip

At the heart of the proofs is the analysis of certain spectral properties of the linear operators 
$-\Delta$, $\partial_{tt} - \Delta + m$ and $\im \partial_t - \Delta + m$ acting on (possibly time quasi-periodic) functions defined on the flat torus $\T^d_\sL $.  The eigenvalues of  $ - \Delta_{\sL}  $  are
\begin{equation}
\label{muj}
\mu_j := \norm{W j }^2   \, , \ \  j \in \Z^d \, , 
\end{equation}
where $ \norm{ y}^2 := \sum_{a=1, \ldots, d} y_a^2  $ is the euclidean norm of $ \R^d $ 
and  $ W  $ is the $d \times d$ invertible matrix 
\begin{equation}
\label{VeW}
W  := {\lV}^{-  \top} \, , \quad \lV := (\bv_1 | \cdots | \bv_d) 
\end{equation}
(with $ {\lV}^{-  \top} $ the transpose matrix of $ {\lV}^{- 1}  $) {associated to 
the dual lattice}. 
Since $ W $ is invertible we clearly have that $ \mu_j \sim \norm{ j }^2 $. 
In the particular case of a rectangular torus 
$$
 W = {\rm diag}(\nu_1, \ldots, \nu_d)  \, , \ \ \nu_a:= \tL_a^{-1} \, ,  \quad 
\mu_j 
= \sum_{a=1}^d \nu_a^2 \,  j_a^2  \, , 
\quad j_a \in \Z \,  ,  \quad j := ( j_a)_{a=1, \ldots, d} \, .
$$
We show that $\Z^d$ may be partitioned into dyadic clusters  (i.e.   $\max |j| \leq 2 \min |j| $ for all integers in the same cluster), 
such that if $j, j' \in \Z^d$ are in different clusters, then the pairs $(j, \mu_j)$, $(j', \mu_{j'})$  are well separated, {see 
 Theorem \ref{thm:sep}.  This holds true for any flat torus. 
This result extends Bourgain cluster decomposition  \cite{bou98}  to any flat torus (in particular rectangular ones).  
In order to obtain such clusterization, we prove  that a chain of distinct integer vectors
of $ \Z^d $ for which the corresponding pairs $(j, \mu_j)$ have uniformly bounded distance $ \Gamma $
contains at most $ \sim \Gamma^{C_1(d)} $ elements, see Proposition \ref{prop:lG}.

Concerning $\partial_{tt} - \Delta + m$ and $\im \partial_t - \Delta + m$, 
restricted to act on  subspaces of functions on $ \T^d_\sL $ which are quasi-periodic in time 
with Diophantine frequency vector $\omega \in \R^n $,  we show that their 
``singular" sites form time-space clusters which are ``well-separated" as well, namely 
their eigenvalues  
\begin{equation}
\label{mujl}
- (\omega \cdot \ell)^2 + \mu_j + m \, , \quad - \omega \cdot \ell + \mu_j + m \, , \qquad \ell \in \Z^n \, , \  j \in \Z^d \, , 
\end{equation}
which are in modulus less than $ 1 $, form 
chains with a bounded controlled length.  This means, roughly speaking,  that any sequence $(\ell_q, j_q)_{q=0, \ldots, L}\subset \Z^n \times \Z^d$ of indices such that (i) the correspondingly eigenvalues \eqref{mujl} have all modulus smaller than 1 and (ii) the distances between two consecutive indices are uniformly bounded, has necessarily a length $L$ which is bounded  in an appropriate quantitative way, 
{see Propositions \ref{thm:lcw} and \ref{thm:lcn}}. Notice that in the first case in \eqref{mujl} (corresponding to NLW)
 the singular sites $ \ell, j $ stay near a ``cone" (deformed by the lattice),  
while in the second  one (corresponding to NLS)  near a ``paraboloid".

The key argument to prove these separation properties consists, 
following the lines of  \cite{Bou05,bebo12}, to show that   certain bilinear forms,
 naturally associated to the quantities \eqref{muj}, \eqref{mujl} (see  \eqref{def:phi}, \eqref{def:phiNLW}), are nondegenerate when restricted to the 
finite dimensional subspace spanned by chains of singular sites. 
This allows to estimate the projections of the vectors of any chain on the subspace generated by the
vectors of the chain it-selves. If such a subspace has maximal dimension, we   deduce 
a bound for all the vectors of the chains. Otherwise, by an inductive argument on the dimension of such 
subspace, we deduce a bound for the  length of the chain. 
In case of \eqref{muj}, the associated bilinear form is  positive definite 
with a quantitative lower bound, for any flat torus, 
see Lemma \ref{detA}. In the first case  \eqref{mujl} (corresponding to NLW) 
the bilinear form has signature, but nevertheless the finite dimensional restrictions \eqref{defA}
we care about are nondegenerate {for any flat torus $ \T^d_\sL $, 
for almost all frequency vectors $\omega = \lambda \bar \omega $} 
with a prescribed  Diophantine direction $ \bar \omega $, see Lemma  \ref{lem:invW}
(in \cite{Bou05,bebo12, Wpre} further quadratic or higher order Diophantine conditions are required).
This is a consequence of the careful expansion \eqref{Pdef} and the invertibility of the compound matrices formed by the minors  of the invertible matrix $ W $ of the lattice. 

Once the spectral properties discussed above are proven, we apply Delort abstract theorem \cite{del2} to control the growth of Sobolev norms for \eqref{sch}, and Berti-Corsi-Procesi abstract Nash-Moser theorem \cite{BCP} to construct quasi-periodic solutions for \eqref{NLWS}.

\smallskip
Finally let us remark that in the last few years there has been an increase {of interest in studying the dynamics of Schr\"odinger  equations on  rectangular  rational or irrational tori.
 These papers can be roughly divided into two groups: the first one concerns Strichartz estimates and well posedness results \cite{bou07, catoire10, guo-oh, bourgain15, deng17}. The second group   analyzes phenomenon of growth of Sobolev norms 
 \cite{deng-germain17, deng217, staffilani_wilson18}, see 
also   \cite{visciglia17}  for $ 2 $ and $ 3 $ dimensional arbitrary compact  manifolds.
It would be interesting to see if the techniques developed in this paper could allow to extend 
these results for flat tori, as well as those in
\cite{maspero_procesi, GHHMP} concerning 
stability and instability of finite gap solutions of NLS 
on the {standard} torus}.


We state now more precisely our results.

\subsection{Growth of Sobolev norms for  linear, time dependent Schr\"odinger equations}

Rescaling the  spatial variables,  \eqref{sch} 
is equivalent to the following Schr\"odinger equation on the standard torus
\begin{equation}
\label{isch}
\im \partial_t \psi = - \Delta_\sL \psi + V(t, x) \psi  \, ,  \quad 
x \in \T^d\,  , 
\end{equation}
 with  the anisotropic Laplacian
$$
\Delta_\sL := \sum_{a,b = 1}^d  \partial_{x_a} [W^\top W]_{a}^b  \partial_{x_b} =
\sum_{a,b,l = 1}^d  W^a_l W_{l}^b  \partial^2_{x_a x_b}   \, ,
$$
where $ W = (W_a^b)_{a,b=1, \ldots, d} $ is the matrix \eqref{VeW} associated to the {dual} lattice  
(with row index $a$ and  column index $ b $).
In the particular case of  the  rectangular torus \eqref{torus} 
the anisotropic Laplacian reduces to 
$$
\Delta_{\vec \nu} := \sum_{a=1}^d \nu_a^2 \,  \partial_{x_a}^2 \, , \qquad \nu_a = \tL_a^{-1} \, , \qquad 
\vec \nu := ( \nu_a)_{a=1, \ldots, \nu} \, . 
$$
The eigenvalues of  $\Delta_\sL  $ are $ - \mu_j$ with $ \mu_j $ defined in  \eqref{muj}.
In \eqref{isch}, for  simplicity of notation, 
 we denoted again by $V$ the  potential in the new rescaled variables, so   that $V(t,x)  $ is a function in 
 $ C^\infty(\R \times \T^d,   \R)$.
As phase space we consider  Sobolev spaces of periodic functions
$$
H^r = \Big\{ \psi = \sum_{j \in \Z^d} \psi_j e^{{\rm i} j \cdot x }  \in L^2(\T^d, \C) \,  
\colon \norm{\psi}_r^2 := \sum_{j \in \Z^d} \la j \ra^{2r} \abs{\psi_j}^2 < \infty \Big\}
$$
where $ \la j \ra := \sqrt{1+|j|^2} $.

Our first result is an upper bound on the speed of growth of Sobolev norms of \eqref{isch}. 
\begin{theorem}[Growth of Sobolev norms]\label{thm:main}
Consider the Schr\"odinger equation \eqref{isch} with potential 
 $V $ in $ C^\infty(\R \times \T^d, \R)$ satisfying 
\begin{equation}
\label{V.prop}
\sup_{ (t,x)\in \R\times \T^d }\abs{\partial_t^\ell \partial_x^{\alpha} V(t,x)} \leq C_{\ell,\alpha} \, , \quad \forall \ell \in \N \, , \ \  \forall \alpha \in \N^d . 
\end{equation}
Then,  for any $r >0$, for any $\epsilon >0$, there exists a constant $C_{r,\epsilon} >0$ such that each solution $\psi(t)$ of  \eqref{isch} with initial datum $\psi_0 \in H^r$ fulfills
\begin{equation}
\label{growth}
\norm{\psi(t)}_r \leq C_{r,\epsilon} \la t \ra^\epsilon \, \norm{\psi_0}_r . 
\end{equation}
\end{theorem}

The proof of Theorem \ref{thm:main} is based on a result of  ``clustering" for 
the eigenvalues of $ - \Delta_{\sL} $  which was proved originally by Bourgain 
 \cite{bourgain99} (see also \cite{bou98, Bou05}) for the Laplacian on the standard torus $ \T^d $, 
and that we extend here to arbitrary flat tori $ \T^d_{\sL} $, non just rectangular, 
see Theorem \ref{thm:sep}. It is interesting that such a result holds for {\it any flat} torus, 
 the ultimate reason is that {Lemma \ref{detA}} holds for any lattice $ \sL $. 
 
Bourgain \cite{bourgain99} used this decomposition result  to  prove  
a $\la t\ra^\epsilon$  upper
 bound for the growth in time of the Sobolev norms of 
the solutions of the linear Schr\"odinger equation \eqref{sch} on $\T^d$, 
in presence of a smooth potential with arbitrary  time dependence.
We also mention that,  
if the potential is analytic and quasi-periodic in time, 
Bourgain \cite{bou99}  proved a logarithmic growth for the solutions of \eqref{sch}
 in $d = 1 $, and, in $ d = 2 $, under a  smallness assumption on the potential. 
Using KAM techniques, Eliasson-Kuksin \cite{EK1} proved  that, 
for an analytic small potential and  a large set of frequencies, the
linear Schr\"odinger equation \eqref{sch} on $ \T^d $ can be 
conjugated to a block-diagonal dynamical system whose
 solutions  have bounded Sobolev norms. 

The proof of the result of Bourgain \cite{bourgain99} was greatly simplified by Delort \cite{del2} who proposed  an abstract  framework which allows to prove a $ \la t\ra^\epsilon $ bound  
{also on other manifolds, for example Zoll ones}; see also \cite{fang} for logarithmic bounds in case of potentials  on $\T^d$ with Gevrey regularity. 

It is the abstract result of Delort \cite{del2}, combined with the novel clustering of the eigenvalues of
$ - \Delta_\sL $ on $  \T^d_\sL  $ in Theorem \ref{thm:sep}, 
 that we employ to deduce Theorem \ref{thm:main}.

\smallskip

Upper bounds of the form \eqref{growth} have a long history.
The first results are due to Nenciu \cite{nen}, who proved, in an abstract framework,  a $ \la  t \ra^\epsilon$ upper bound for the expected value of the energy  in case the system has  increasing spectral gaps and bounded perturbation, and by 
Duclos,  Lev  and \v S\v tov\'\i \v cek \cite{duclos} in case of decreasing spectral gaps. 

Such results have been recently improved,  in an abstract setup,  by  Maspero-Robert  \cite{maro}  
for some unbounded perturbations, and by Bambusi-Gr\'ebert-Maspero-Robert \cite{BGMR2} for a larger class of unbounded perturbations and 
 without spectral gaps assumptions. 
Unfortunately these results do not apply to the Schr\"odinger equation {\eqref{sch}} on $ \T^d $.

Finally we mention that 
potentials $V(t,x ) $ which provoke  Sobolev norms explosion have been constructed 
by  Bourgain  \cite{bou99} 
for a Klein-Gordon and Schr\"odinger equation on $\T$, by Delort \cite{del} 
for the harmonic oscillator on $\R$, by Bambusi-Gr\'ebert-Maspero-Robert  \cite{BGMR1} 
for the Harmonic oscillators on $\R^d$, $d \geq 1$, and by Maspero \cite{ma18}
for systems enjoying an asymptotically constant spectral gap condition.

\subsection{Quasi-periodic solutions for NLW and NLS equations}
As before, we first rescale the spatial variables, recasting \eqref{NLWS} to  the 
nonlinear wave equation (NLW) and nonlinear Schr\"odinger equation (NLS)
 with anisotropic Laplacian  on the torus $\T^d$, 
\begin{equation}
\label{NLWa}
u_{tt} - \Delta_\sL u + m u = \epsilon f(\omega t, x, u) \, ,
\qquad 
\im u_t - \Delta_\sL u + m u = \epsilon \tf(\omega t, x, u) \,  , 
 \quad
x \in \T^d \, .
\end{equation}
Concerning regularity we assume that the nonlinearity $f \in C^\tq(\T^n \times \T^d \times \R; \R)$, resp. 
$\tf \in C^\tq(\T^n \times \T^d \times \C; \C)$ in the real sense (namely as a function of ${\rm Re}(u)$, ${\rm Im}(u)$),
for some $\tq$ large enough.
We also require that 
\begin{equation}
\label{tfham}
\tf(\omega t, x, u) = \partial_{\bar u} F(\omega t, x, u) \in \R \,  , \qquad \forall u \in \C \, , 
\end{equation}
so that the NLS equation is Hamiltonian.

Concerning the frequency $\omega \in \R^n$, we constrain it to  a fixed direction, namely
\begin{equation}
\label{omega}
\omega = \lambda \bar \omega \, , \qquad \lambda \in \Lambda := [1/2, 3/2] \, , \qquad \abs{\bar \omega}_1:= \sum_{p=1}^n |\bar\omega_p| \leq 1 \, . 
\end{equation}
The vector $\bar\omega \in \R^n$ is assumed to be Diophantine, i.e.  for some $\tau_0 \geq n$, 
\begin{equation}
 \label{omega1}
 \abs{ \bar \omega \cdot \ell} \geq \frac{2 \gamma_0}{|\ell|^{\tau_0}} \, , \quad \forall \ell \in \Z^n\setminus \{0\} \, ,
 \end{equation} 
where, here and below
 $| \ell | := |\ell|_\infty = \max(|\ell_1|, \ldots, |\ell_n|)$.

The search for quasi-periodic solutions of \eqref{NLWa} reduces to finding solutions $u(\vf, x)$, 
periodic in all the  variables $(\vf, x) \in \T^n \times \T^d$,  of the equations
\begin{equation}
\label{NLWb}
(\omega\cdot \partial_\vf)^2 u - \Delta_\sL u  + m u  = \epsilon f(\vf, x, u ) \, , 
\qquad 
\im \omega \cdot \partial_\vf u - \Delta_\sL u + m u = \epsilon \tf(\vf, x, u) \, , 
\end{equation}
with $u$ belonging to some Sobolev space $H^s(\T^n \times \T^d)$ of 
real valued functions, in the case of NLW,  respectively complex valued for NLS.

\begin{theorem}[Quasi-periodic solutions of NLW and NLS]
\label{main2}
Consider the NLW equation \eqref{NLWa}  or the NLS equation \eqref{NLWa}-\eqref{tfham}
and assume \eqref{omega}-\eqref{omega1}.
Then there are $s,  \tq \in \R $ 
such that, for any $ f  $, $ {\mathtt f} \in C^{\tq}$ and
for all $\e\in[0,\e_{0})$ with $\e_{0}>0$ small enough, there is a map
$$
u_\e\in C^{1}([1/2,3/2],H^{s}(\T^n \times \T^d)) \, ,\qquad
\sup_{\lambda \in [1/2,3/2]} \|u_\e (\lambda)\|_{H^{s}(\T^n \times \T^d)}\to 0,\mbox{ as }\e\to0 \,  ,
$$
and a Cantor-like set $\cC_{\e}\subset[1/2,3/2]$, satisfying $ {\rm meas(\cC_{\e})}\to1\mbox{ as }\e\to 0 $, 
such that, for any $\lambda \in\cC_{\e}$, $u_\e (\lambda)$ is a solution of
\eqref{NLWb}, with $\omega=\lambda \bar \omega $.
\end{theorem}

In order to prove Theorem \ref{main2} we verify that the assumptions of the abstract Theorems 2.16--2.18 of   \cite{BCP} are met. For the reader convenience, we report such results in Appendix \ref{app:BCP}.
The key property to prove is that the singular sites of the 
quasi-periodic  wave and Schr\"odinger
operators
$ (\omega\cdot \partial_\vf)^2  - \Delta_\sL   + m $ and
$ \im \omega \cdot \partial_\vf  - \Delta_\sL  + m  $, 
 form clusters in space-time Fourier indices which are sufficiently separated. We prove 
 properties of this kind 
 for any flat torus $ \T^d_\sL $ in sections \ref{thm3.1} and \ref{sec41}.

Existence of quasi-periodic solutions for analytic nonlinear Schr\"odinger equations with a convolution potential
was first proved by Bourgain \cite{bou98} on $ \T^2 $    
and then extended in \cite{Bou05} to arbitrary dimension and for wave type pseudo-differential equations.
The convolution potential is used as a parameter to fulfill suitable non-resonance conditions. 
The proof is based on a multiscale approach, originated by Anderson localization  theory, to invert 
approximately the linearized operators arising at each step of a Newton iteration.
This approach requires ``minimal" non-resonance conditions, and separation properties for the clusters of singular sites. 
Subsequently existence of quasi-periodic solutions of NLS with  convolution potential 
which are also linearly stable has been proved by Eliasson-Kuksin
\cite{EK10}, who managed to reduce the linearized equations to a block-diagonal dynamical system,
imposing the stronger second order Melnikov non-resonance conditions.
For the completely resonant NLS, which has no external parameters, we refer to  Wang \cite{WDuke} and 
Procesi-Procesi \cite{PP15, pp_bumi} who also proved  the linear stability of the torus; see also 
\cite{EGK} for the beam equation. Unfortunately no reducibility results for NLW on $ \T^d $ are available. 

Theorem \ref{main2} extends the results in \cite{bebo12, bebo13} valid for a square torus $ \T^d $
to the case of an arbitrary flat torus $ \T^d_\sL  $. The technique  for proving the separation properties
of the singular sites
in sections  \ref{thm3.1} and \ref{sec41} 
could  also be applied to extend 
the existence results of quasi-periodic solutions formulated as in \cite{Bou05}.
We find interesting  that,  with this multiscale approach, the  
irrationality properties of the lattice generators $ \sL $ do not play any role. Such properties
 could be relevant for proving also reducibility results, 
 i.e. for imposing second order Melnikov non-resonance conditions. 
We also mention that Theorem \ref{main2} 
improves, in the case of NLW,  the results in \cite{bebo12, BCP}, 
because it requires {$ \bar \omega $ } to satisfy only the standard Diophantine condition 
\eqref{omega1}, and not  additional quadratic Diophantine conditions. 
This is due to the more careful analysis  in section \ref{sec:5.3} 
 to verify  separation properties of the singular sites. 

\medskip
\noindent{\bf Acknowledgments.} 
The authors thank T. Kappeler for posing the question for flat tori. 
During the preparation of this work we were partially supported by  Prin-2015KB9WPT and  Progetto  GNAMPA - INdAM 2018 ``Moti stabili ed instabili in equazioni di tipo Schr\"odinger''.

\section{Growth of Sobolev norms}

As we anticipated in the introduction, the key step 
for applying Delort theorem \cite{del2} is to prove the following 
result about clusterization properties of the eigenvalues $ \mu_j $ 
of $ - \Delta $ on $ \T^d_\sL $, i.e. of $ - \Delta_\sL $ on $ \T^d $. 
\begin{theorem}[Clustering of eigenvalues of $ - \Delta_\sL $]
\label{thm:sep}
There exist constants $ \delta_0 (d) >0$, $ C( \sL,d)>0$, 
$ \tC (\sL, d) > 0 $, $\tc(d) >1 $ such that, for any  $ 0 < \delta < \delta_0 (d) $,
 there exists a  partition 
$(\Omega_\alpha)_{\alpha \in A}$ of $\Z^d$ satisfying:  
\begin{itemize}
\item[(1)]  $\forall \alpha \in A$,  $\forall j_1, j_2 \in \Omega_\alpha$, we have 
\begin{equation}
\label{cl1}
\abs{j_1 - j_2} + \abs{ \mu_{j_1} - \mu_{j_2}} \leq C( \sL,d) \left( |j_1| + |j_2|\right)^{\tc(d) \delta } \, . 
\end{equation}
Moreover each $ \Omega_\alpha $ is finite and, either $\max_{j \in \Omega_\alpha} |j| \leq \tC(\sL, d)$  or  $\max_{j \in \Omega_\alpha} |j| \leq 2 \min_{j \in \Omega_\alpha} |j|$.

\item[(2)] $ \forall \alpha, \beta \in A $, 
$ \alpha \neq \beta$, $\forall j_1 \in \Omega_\alpha$, $\forall j_2 \in \Omega_\beta$, we have 
\begin{equation}
\label{cl2}
\abs{j_1 - j_2} + \abs{ \mu_{j_1} - \mu_{j_2}} > \left( |j_1| + |j_2|\right)^{\delta}  \, . 
\end{equation}
\end{itemize}
\end{theorem}

We postpone the proof of this result to  section \ref{sec:L}, and 
we now show that all the assumptions of  Delort theorem \cite{del2} are met.
We begin with some preliminary notation.
For $j \in \Z^d$, denote by $\Pi_j$ the spectral projector
\begin{equation}\label{expo}
{\Pi_j u  := \la u, {\mathtt e}_j \ra_{L^2}  {\mathtt e}_j   \, , \quad 
{\mathtt e}_j  := \frac{e^{\im j \cdot x}}{( 2\pi)^{d/2} }} \, .
\end{equation}

\begin{definition}\label{def:2}
For any $\sigma \in \R$ we denote by $\cA^\sigma$ the space of smooth in time families  of continuous operators 
{$ Q(t)$} from $C^\infty(\T^d)$ to $\cD'(\T^d)$ such that, $\forall k, N \in \N$, there is $C_{k,N} >0$ such that
$$
\norm{\Pi_j \, \partial_t^k Q(t)\,  \Pi_{j'}}_{\cL(L^2)} \leq C_{k,N} \frac{(1 + |j| + |j'|)^\sigma}{\left(1 +  |j - j'|\right)^N} \, , \quad \forall j, j' \in \Z^d , \ \  \  \forall t \in \R \, .
$$
We denote $\cA^{- \infty} := \bigcap_{\sigma \in \R} \cA^\sigma$.
\end{definition}
We remark that each $ Q $ in $\cA^\sigma$ extends to a smooth family in $ t $ 
of bounded operators from $ H^s (\T^d) $ 
to $ H^{s-\sigma} (\T^d)$. Moreover, 
if $Q \in \cA^\sigma$, then its adjoint $Q^* \in \cA^\sigma$ (the adjoint is with respect to the standard $L^2(\T^d)$ scalar product).
Moreover, for any  $\sigma_1, \sigma_2 \in \R$ one has  $\cA^{\sigma_1} \circ \cA^{\sigma_2} \subseteq \cA^{\sigma_1 + \sigma_2}$ and,  if $\sigma_1 \leq \sigma_2$, then $\cA^{\sigma_1}\subseteq \cA^{\sigma_2}$.
Finally, if the potential $V \in C^\infty(\R \times \T^d, \R)$ fulfills \eqref{V.prop}, then, 
as a multiplication operator, $V \in \cA^0$.

\smallskip
We now state Delort theorem \cite{del2}, specified to the case of the torus.
\begin{theorem}[Delort]
Assume that for any $\sigma \in \R$, there exist subspaces $\cA^\sigma_D$, $\cA^\sigma_{ND}$ of $\cA^\sigma$, invariant under $Q \mapsto Q^*$, such that the following holds true:
\begin{itemize} 
\item[(H1)] For any $Q \in \cA^\sigma$, there exist $Q_D \in   \cA^\sigma_{D}$ and $Q_{ND} \in  \cA^\sigma_{ND}$ such that $Q = Q_{D} + Q_{ND}$.
\item[(H2)] There is $\rho >0$ such that for any $\sigma \in \R$, any  $W \in \cA^\sigma_{ND}$, there exist $X \in \cA^{\sigma-\rho}_{ND}$ and $R \in \cA^{ -\infty}$ 
solving the homological equation
\begin{equation}
\label{hom.eq}
 [  X, -\Delta_\sL  ]   = W + R \, . 
\end{equation}
Here $[A,B]:= AB- BA$ is the commutator of the operators $A,B$.
\item[(H3)] There is an element $\cD \in \cA^2$, independent of time and commuting with $\Delta_\sL$, which defines an equivalent norm on $H^r(\T^d)$, {i.e. for some constants  $ c, C >0$} 
\begin{equation}
\label{eq.norm}
c \norm{u}_{H^r(\T^d)} \leq  \norm{ \cD^{r/2} u }_{L^2(\T^d)} \leq C \norm{u}_{H^r(\T^d)} \, , \quad \forall u \in H^r(\T^d) \, ,
\end{equation}
and, furthermore,
\begin{equation}
\label{QD}
[Q, \cD] \in \cA^{- \infty}  , \quad \forall Q \in \cA^\sigma_{D} \, . 
\end{equation}
\end{itemize}
Then, for all $ r, \epsilon >0$,  there exists a constant $C_{r,\epsilon} >0$ such that each solution $\psi(t)$ of 
the linear Schr\"odinger equation \eqref{isch} with initial datum $\psi_0 \in H^r(\T^d)$ satisfies 
$$
\norm{\psi(t)}_r \leq C_{r,\epsilon} \la t \ra^\epsilon \, \norm{\psi_0}_r . 
$$
\end{theorem}

\paragraph{Verification of the assumptions of Delort theorem.}

Consider  the cluster decomposition $(\Omega_\alpha)_{\alpha \in A}$ provided by Theorem \ref{thm:sep} and 
define {the  projector} 
$$
\wt\Pi_\alpha := \sum_{j \in \Omega_\alpha} \Pi_j \,  , \qquad \forall \alpha \in A \, . 
$$
Following \cite{del2}, 
we denote by $\cA^\sigma_{D}$ the {subspace} of $\cA^\sigma$ given by those operators $Q \in \cA^\sigma$ such that
$\wt\Pi_\alpha Q \wt\Pi_\beta = 0$ for any 
$\alpha, \beta \in A$ with $ \alpha \neq \beta$. 
Similarly, we denote by $\cA^\sigma_{ND}$ the operators satisfying 
$\wt\Pi_\alpha Q \wt\Pi_\alpha = 0 $ for any $ \alpha \in A$.
Clearly such subspaces are invariant under $Q \mapsto Q^*$.

We verify now that (H1)--(H3) are met. 
\\[1mm]
\underline{Verification of (H1).} It is enough to write
$Q = \sum_{\alpha} \wt\Pi_\alpha Q \wt\Pi_\alpha + \sum_{\alpha \neq \beta} \wt\Pi_\alpha Q \wt\Pi_\beta 
\equiv Q_D + Q_{ND}$.
\\[1mm]
\noindent\underline{Verification of (H2).} 
For any $Q \in \cA^\sigma$, denote by $ {Q_j^{j'}} $ its matrix elements  on  the exponential basis 
$ {\mathtt e}_j $ defined in \eqref{expo}, i.e. 
$ {Q_j^{j'}} := \la Q {\mathtt e}_{j'} , \, {\mathtt e}_{j} \ra_{L^2}$. 
{Thus,} 
the homological equation \eqref{hom.eq} reads
$$
{ (\mu_{j'} - \mu_{j}) X_j^{j'}} = {W_j^{j'}} + 
{R_j^{j'}}  \, , \quad j \in \Omega_\alpha, \ j' \in \Omega_\beta \, , \ \alpha \neq \beta \, ,
$$
{and  we define its solution} 
 $$
{X_j^{j'}} := 
\begin{cases}
\displaystyle{\frac{ {W_j^{j'}} }{{\mu_{j'} - \mu_j}}}  & \mbox{ if } |\mu_{j} - \mu_{j'}| \geq \frac{1}{4} (|j| + |j'|)^\delta \\
0  & \mbox{ otherwise}
\end{cases} , 
$$
with
$$
{R_j^{ j'}} := 
\begin{cases}
- {W_j^{j'}}  & \  \  \ \mbox{ if }  |\mu_{j} - \mu_{j'}| < \frac{1}{4} (|j| + |j'|)^\delta\\
0 & \ \ \ \mbox{ otherwise}
\end{cases} . 
$$
{Since $W \in \cA^\sigma_{ND}$ we deduce that}
$X \in \cA^{\sigma- \delta}_{ND}$.  To verify that $R \in \cA^{- \infty}$, recall that Theorem 
\ref{thm:sep}-(2) implies that  
$|j - j'| + |\mu_{j} - \mu_{j'}| \geq  (|j| + |j'|)^\delta$, 
$ {\forall  j \in \Omega_\alpha} $, 
$ { j' \in \Omega_\beta} $, $ { \alpha \neq \beta} $, 
and therefore  $( {R_j^{j'}} )_{j, j' \in \Z^d}$
 is supported on  sites satisfying 
$|j - j'| \geq \frac{1}{2} (|j| + |j'|)^\delta $. As a consequence, recalling Definition \ref{def:2}, 
the operator  $R $ is in $ \cA^{m}$ for any $m \in\R$. 
In conclusion  (H2) holds with $\rho = \delta$.
\\[1mm]
\noindent\underline{Verification of (H3).}
Define the operator
$$ \cD := \sum_{\alpha \in A} M_\alpha^2 \, \wt\Pi_\alpha \,   , \qquad 
M_\alpha:= \max_{j \in \Omega_\alpha}|j| \, ,  
$$
{(the maximum $M_\alpha$  is attained at some $ j(\alpha) $ since $\Omega_\alpha$ is finite)}. 
{Clearly 
 $ \cD $ commutes with $ \Delta_\sL $ 
 and also with any operator $Q \in \cA^\sigma_D$, so \eqref{QD} is trivially satisfied.
By Theorem \ref{thm:sep}-(1) each cluster $ \Omega_\alpha  $ is dyadic, i.e. 
$$
\forall \alpha \in A \, , \  \forall j \in \Omega_\alpha \, , \quad
\mbox{either } \  |j| \leq \tC (\sL, d) \quad \mbox{ or } \  \frac{M_\alpha}{2} \leq |j| \leq  M_\alpha   \, ,
$$
and therefore the operator $ \cD $ is in $ \cA^2 $ and \eqref{eq.norm} holds}.

\subsection{Proof of Theorem \ref{thm:sep}}\label{sec:L}

Denote by 
$\nla{\cdot}:\R^d \to \R$  the  quadratic positive definite  form
\begin{equation}
\label{def:cQ}
\nla{y} := \norm{W y}^2 ,   \quad \forall y  \in \R^d \, , 
\end{equation}
and  by  $\pla{\cdot}{\cdot}:\R^d \times \R^d \to \R$ the associated scalar product 
\begin{equation}
 \label{def:phi}
 \pla{y}{y'} := \la W y,
 W y'  \ra_{\R^d}  \, , \quad \forall  y, y' \in \R^d \, ,
 \end{equation} 
where  $ \la  \cdot, \cdot   \ra_{\R^d} $ 
denotes the standard scalar product of $ \R^d $.
Note that,
for any $j \in \Z^d$,  it results that 
$ \nla{j} = \pla{j}{j} = \mu_j$ are the eigenvalues \eqref{muj} of $ - \Delta_{\sL} $.

 The proof of Theorem \ref{thm:sep} is based on a bound for the length of a 
 $\Gamma$-chain, which we now introduce.
Define the map
\begin{equation}\label{def:Phi}
\Phi: \Z^d \to \Z^{d+1} \, , \qquad \Phi(j) := (j, \nla{j}) = (j, \mu_j ) \, .
\end{equation}

\begin{definition}[$\Gamma$-chain]\label{1Gchain}
 Given $\Gamma \geq 2$, a sequence of distinct 
integer vectors $( j_q)_{q=0, \ldots, L} \subset \Z^d$ is   a 
{\em $\Gamma$-chain} if
$$
\abs{\Phi(j_{q+1}) - \Phi (j_q)} \leq \Gamma \, , \quad \forall 0\leq q \leq L-1 \, . 
$$ 
The number $L$ is called the {\em length} of the chain.
\end{definition}

The next proposition provides  the key bound for  the length of a $\Gamma$-chain.

\begin{proposition}[Length of $\Gamma$-chains] 
\label{prop:lG}
There exist $C_1( d), C_2(\sL, d) >0$ such that any $\Gamma$-chain  
$( j_q)_{q=0, \ldots, L} \subset \Z^d$ has length at most
\begin{equation}
\label{lG}
L \leq  C_2(\sL, d) \  \Gamma^{C_1( d)} \, . 
\end{equation}
\end{proposition}

The proof of Proposition  \ref{thm:sep} is based on an inductive argument on the dimension of the subspace 
$ G \subset  \R^d $ generated by the vectors $ j_q - j_{q'} $, for all $ q, q' \in \{0, \ldots, L \} $, see \eqref{def:Gs}. 
First of all, by the Definition \ref{1Gchain} of a $ \Gamma $-chain we derive the following 
bounds for the scalar product   of each vector $ j_{q_0} $
with the elements of the subspace $ G $. 

\begin{lemma}
For all $ q, q_0 \in \{0, \ldots, L\}$ we have 
\begin{equation}
\label{G.prop}
\abs{\Phi(j_q) - \Phi(j_{q_0})} \leq \abs{q-q_0} \Gamma \, , \qquad
\abs{\pla{j_{q_0}}{j_q - j_{q_0}}} \leq C(\sL, d) |q-q_0|^2 \Gamma^2 \, . 
\end{equation}
\end{lemma}

\begin{proof}
The first bound in \eqref{G.prop} directly follows by the definition of $ \Gamma $-chain.
By bilinearity, $\nla{j_q} =  \nla{j_{q_0}}  +2 \pla{j_{q_0}}{j_q - j_{q_0}} +  \nla{j_q - j_{q_0}}$, and thus
\begin{align*}
\Phi(j_q) - \Phi(j_{q_0}) 
& =   \Big(j_q - j_{q_0}, \  2 \pla{j_{q_0}}{ j_q - j_{q_0}} + \nla{j_q - j_{q_0}}  \Big) \, .
\end{align*}
The first bound in \eqref{G.prop} and the estimate $\nla{y} \leq  C(\sL, d) |y|^2$, imply 
that, for all $  q, q_0 \in \{0, \ldots, L\}$, 
$$
\abs{\pla{j_{q_0}}{ j_q - j_{q_0}}} \leq |q-q_0| \Gamma +  C(\sL, d) |q-q_0|^2 \Gamma^2 \leq C(\sL, d) |q-q_0|^2 \Gamma^2 \, , 
$$
which is the second estimate in \eqref{G.prop}.
\end{proof}

Now we introduce the subspace of $ \R^d $:
\begin{equation}\label{def:Gs}
\begin{aligned}
\sub &:= {\rm span }_\R \la j_q - j_{q'} \, , \ \  0 \leq q, q' \leq L \ra \\
& \ =  {\rm span }_\R \la j_q - j_{q_0} \, , \ \  q = 0, \ldots, L \ra \, , \ \forall q_0 = 0, \ldots, L \, . 
\end{aligned}
\end{equation}
Let $ g := \dim \sub $; 
clearly $ 1 \leq g \leq d$ (note that $g\geq 1 $ because  the vectors $j_q$ are all distinct).
We consider the orthogonal decomposition  of $ \R^d $ with respect to the scalar product $\pla{\cdot}{\cdot}$, 
\begin{equation}\label{GGbot}
\R^d = \sub \oplus \sub^{\bot \pla{\cdot}{\cdot}} \, ,
\end{equation}
and we denote by $P_\sub $  the orthogonal projector  on the subspace $ \sub $. 

The next step is to derive by \eqref{G.prop} a bound on the norm of  $P_\sub j_{q_0}$ for each 
integer vector $ j_{q_0} $, $ q_0 = 0, \ldots, L $, of the $ \Gamma $-chain.  
We consider two cases.
\\[1mm]
{\em Case 1.} For any  $q_0 \in \{0, \ldots, L\}$ it results that 
$ {\rm span }_\R \la j_q - j_{q_0} \, ,  \ \ \abs{q-q_0} \leq L^\varsigma \ra = \sub $ where $ \varsigma  := [3(2d+1)d]^{-1} $.
We select a basis of $ \sub $ from $j_q-j_{q_0}$ with $|q-q_0|\leq L^\varsigma $, say
\begin{equation} \label{def:fi}
f_i := j_{q_i} - j_{q_0}  \  \ 
\mbox{ with } \  \  |q_i-q_0|\leq L^\varsigma ,  \quad 
 \forall 1 \leq i \leq g \, .
\end{equation}
 By the first bound in \eqref{G.prop} and recalling \eqref{def:Phi} we obtain the estimate
\begin{equation}
\label{est.fj}
\abs{f_i} \leq |q_i-q_0| \Gamma \leq L^\varsigma \Gamma \, , \qquad \forall 1 \leq i \leq g \, .
\end{equation}
To obtain a bound on the norm of $P_\sub j_{q_0}$, we
decompose $P_\sub j_{q_0}$ on the basis $(f_i)_{1 \leq i \leq g}$, 
\begin{equation}\label{PGjq0}
P_\sub j_{q_0} = \sum_{i=1}^g x_i f_i \  ,
\end{equation}
and we look for an upper  
bound for the coordinates 
$x := (x_i)_{1 \leq i \leq g}$. The $ x $ 
are determined 
by solving the linear system  $ A x = \bb $, where 
\begin{equation}
\label{l.eqA}
 A := \big(\pla{f_i}{f_l}\big)_{1 \leq i, l \leq g} 
\end{equation}
is the $ g \times g $-matrix of the scalar products of the basis $ (f_i)_{i=1, \ldots, g} $ of $ \sub $ (Gram matrix) and
\begin{equation}
\label{l.eq}
\bb := (\bb_l)_{ l=1, \ldots, g}  \in \R^g \, , \quad 
\bb_l := \pla{P_\sub j_{q_0}}{ f_l } =  \pla{j_{q_0}}{ f_l}  \, .
\end{equation}
By the second estimate in \eqref{G.prop} and \eqref{est.fj} the coefficients $\bb_l$ are estimated  by    
\begin{equation}
\label{normb}
|\bb_l|  \leq  c(\sL, d)   |q_l-q_0|^2 \Gamma^2
\leq  c(\sL,d) (L^{\varsigma} \Gamma)^2 \, , \quad \forall l=1, \ldots, g \, .
\end{equation}
Since the vectors $ (f_i )_{i=1, \ldots, g} $ are independent, 
the Gram  matrix $A$ is invertible. 
We prove in Lemma \ref{detA} below 
that  that its determinant is bounded away from zero, uniformly 
with respect to the basis $ (f_i)_{i=1, \ldots, g} $ of integer  vectors. 
We introduce some notation.

\begin{itemize}
\item 
For any  $ g \in \{1, \ldots, d \} $ we denote by 
$ \fC_g(W)$ 
the $g-th$ {\it compound} matrix of the $ d\times d $ matrix $W$, 
defined as the matrix whose entries are the determinants of all possible  $g\times g$ minors of $W$, see e.g. \cite{gantmacher}, Chap. $1$ ${\mathsection 4} $. Thus
$ \fC_g(W)$ is a $ \binom{d}{g} \times \binom{d}{g}  $ square matrix. 
The important result is that, for any $  g $, 
the compound matrix of an invertible matrix $ W$ is invertible as well (see e.g. 
\cite{gantmacher}, Chap. 1 {$\mathsection 4$}) and 
\begin{equation}\label{inv-comp}
\fC_g(W)^{-1} = \fC_g(W^{-1}) \, . 
\end{equation}
\item
If $A$ is a $m \times n$ matrix, and $1 \leq g \leq m$, we denote by 
  $A_{a_1  \ldots a_g}$ the $g\times n$ matrix of rows $a_1,\ldots, a_g$ of $A$. 
  If $1 \leq g \leq n$ then 
 $A^{b_1  \ldots b_g}$ is the $m \times g$ matrix of columns 
$ b_1,\ldots, b_g$ of $A$. Finally if $1 \leq g \leq \min(m,n)$, we denote by $A_{a_1 \ldots a_g}^{b_1  \ldots b_g}$ the $g \times g$ matrix of rows $a_1,\ldots, a_g$ and columns $ b_1,\ldots, b_g$ of $A$.
\end{itemize}

 \begin{lemma}
 \label{detA}
Let  $A$ be the Gram matrix  defined in \eqref{l.eqA}. Then 
\begin{itemize}
\item[(i)] There exists $p \in \Z^{\binom{d}{g}}\setminus \{0\}$ such that  
\begin{equation}
\label{detA1}
\det A = \norm{\fC_g(W) p}^2  
\end{equation}
where  $\fC_g(W)$ is the $g-th$ compound matrix of $W$ and 
$\norm{\cdot}$ is the euclidean norm.
\item[(ii)] 
There exists a constant $\tc(\sL) >0$ such that, for any 
 linearly independent integer vectors  $ (f_i)_{i=1, \ldots, g} $,  
\begin{equation}
\label{deA.e}
\det A \geq \tc(\sL) >  0\,  .
\end{equation}
\end{itemize}
\end{lemma}
\begin{proof}
Recalling \eqref{def:phi} we write the Gram matrix $ A $ in \eqref{defA} as
\begin{equation}\label{WF}
A = \left(W  \, F \right)^\top  \left(W \, F \right) \, , 
\quad  F := \begin{pmatrix}  f_1 \, | \cdots | \, f_g \end{pmatrix}  \, . 
\end{equation}
Then applying twice Cauchy-Binet formula\footnote{
If $M$ is a $g \times d$ matrix, $N$ a $d \times g$ one, then
$$
\det(MN) = \sum_{1 \leq i_1 <  \cdots < i_g \leq d}
\det(A^{ i_1\cdots i_g}) \, \det(B_{ i_1\cdots i_g}) \, .
$$
}
we obtain
\begin{align}
\det A  & = 
\sum_{1 \leq a_1 <  \cdots < a_g \leq d} \Big(\det([W F]_{a_1  \ldots a_g}) \Big)^2 \nonumber 
\\ 
& = \sum_{1 \leq a_1 <  \cdots < a_g \leq d} 
 \left(
  \sum_{ 1 \leq b_1 < \ldots < b_g \leq d} \!\!
  \det(W_{a_1  \ldots a_g}^{b_1 \ldots b_g}) p_{b_1 \ldots b_g}
\right)^2 \label{detA.exp}
\end{align}
where $  p_{b_1 \ldots b_g}:=  \det(F_{b_1  \ldots b_g})  $ are integers  
because  the matrix $ F $ has integer entries. 
The expression 
\eqref{detA.exp} is \eqref{detA1}
 with $ p := (p_{b_1\ldots b_g}) \in \Z^{\binom{d}{g}}$. 
Since the Gram matrix $ A $ of the linearly independent vectors $ (f_i)_{i=1, \ldots, g} $
is invertible, $\det A \neq 0$, and, by \eqref{detA1},
the invertibility of   $\fC_g(W)$ implies that  $p\neq 0$. 
Item ($ii$) follows by item ($i$) and the invertibility  of   $\fC_g(W)$, see \eqref{inv-comp}. 
\end{proof}

As $\pla{f_i}{f_l} \leq c(\sL,d) |f_i| |f_l| \leq c(\sL, d) (L^\varsigma \Gamma)^2$, using 
\eqref{deA.e}, \eqref{est.fj}, 
\eqref{normb},  we estimate $x = A^{-1} \bb$ by Cramer rule, obtaining
\begin{equation}
\label{normx}
|x_i| \leq {c(\sL, d)}\, \left(L^\varsigma \,\Gamma\right)^{2d} , \quad \forall 1 \leq i \leq g \, ,
\end{equation}
and, by \eqref{PGjq0}, \eqref{normx}, 
\eqref{est.fj},   we deduce 
$$
|P_\sub j_{q_0}| \leq C(\sL,d) (L^\varsigma \Gamma)^{2d+1 } \, , \quad \forall q_0 \in \{0,\ldots, L\} \, .
$$ 
In particular,  for all $ q_1, q_2 \in \{0,\ldots, L\}$, we obtain 
\begin{align*}
|j_{q_1} - j_{q_2}| = 
|P_\sub j_{q_1} - P_\sub j_{q_2}| \leq 
| P_\sub j_{q_1}| +  |P_\sub j_{q_2}| & \leq 
2  C(\sL,d) (L^\varsigma \Gamma)^{2d+1} .
\end{align*}
This estimate, and the fact that   $j_q $, $  q = 0, \ldots, L $,  are distinct integers of $ \Z^d $, 
 imply that their number does not exceed
$
 C'(\sL,d) (L^\varsigma \Gamma)^{(2d+1)d}  $
and therefore
$$
L \leq C(\sL,d) (L^\varsigma \Gamma)^{(2d+1)d} \, .
$$
Choosing now $\varsigma$ so small that  $ \varsigma  (2d+1) d \leq 1 /2  $, we obtain \eqref{lG} with $ C_1 (d) = 2 (2d+1) d $.
\\[1mm]
{\em Case 2.}  $\exists q_0 \in \{0, \ldots, L \}$ for which
$$
\sub_1 := {\rm span }_\R \la j_q - j_{q_0} \, ,  \ \ \abs{q-q_0} \leq L^\varsigma \ra \subsetneq \sub \, , 
\quad g_1 := \dim \sub_1  \leq g-1 \, ,
$$
so all vectors $j_q$ with $|q-q_0| \leq L^\varsigma$ belong to an affine subspace  of dimension $1 \leq g_1 \leq g-1$.
Consider then the subchain $\{j_q \colon |q-{q_0}| \leq L^\varsigma \}$, which has a length 
$L_1 \sim L^\varsigma$. If for any index $q_1$  of this subchain it results that  $ {\rm span }_\R \la j_q - j_{q_1} \, ,  \ \ \abs{q- q_1} \leq L_1^\varsigma \ra = \sub_1$, then we can repeat the argument of Case 1 and obtain a bound on $L_1$ (and hence on $L$).
Applying this procedure at most {$d$} times,
 we obtain a bound of the form \eqref{lG}, proving Proposition \ref{prop:lG}.  $\hfill \Box$

\subsection{Conclusion} 

We finally conclude the proof of  Theorem \ref{thm:sep}.
Fix $\delta_0(d) :=   \frac{1}{2C_1(d) +2}$ with $C_1(d)$ given by  Proposition \ref{prop:lG}, take an arbitrary $\delta \in (0, \delta_0(d))$ and  introduce the equivalence relation on $\Z^d$ defined by
\begin{equation}
\label{rel}
j \sim j'   \Leftrightarrow   
\begin{array}{l}
\mbox{either } j = j',  \\
 \mbox{or } 
 \exists  L \in \N  \mbox{ and distinct integer vectors } (j_q)_{q= 0, \ldots, L}\subset \Z^d 
\mbox{ with } j_0 = j \,, 
\\ 
\quad \, j_L = j'  \mbox{ and   } \abs{\Phi(j_{q+1}) - \Phi(j_q)} \leq \left(|j_q| + |j_{q+1}| \right)^\delta,  \  \forall q = 0, \ldots, L-1 \, . 
\end{array}
\end{equation}
This  equivalence relation provides 
a partition of $\Z^d$ in classes of equivalence $(\Omega_\alpha)_{\alpha \in A}$, with  $\Z^d = \cup_{\alpha \in A} \Omega_\alpha$. We claim that such clusters  $ \Omega_\alpha  $ fulfill \eqref{cl1}--\eqref{cl2}.
We split the proof in several steps.
\\[1mm]
\underline{ Each cluster $\Omega_\alpha $ is bounded.} 
By contradiction suppose that $\Omega_\alpha $ is unbounded, i.e. it contains integer vectors of arbitrary large modulus.
For any  $j, j' \in \Omega_\alpha$, 
 with $|j'|$ very large (specified later), 
there exists a sequence $(j_q)_{q=0,\ldots, M}$ satisfying \eqref{rel}. Without loss of generality 
 we  assume that $|j_q| \leq |j'|$, for all $ q$ (otherwise we just replace $j'$ with the $j_{q}$ having  maximum modulus and consider the subchain connecting $j$ and $j_q$). Since 
$$
\abs{\Phi(j_{q+1}) - \Phi(j_q)} \leq \left(|j_q| + |j_{q+1}| \right)^\delta \leq 
(2 |j'|)^\delta \, , \qquad \forall 0 \leq q \leq M-1 \, ,
$$
it follows that $(j_q)_{q=0,\ldots, M}$ is a $(2 |j'|)^\delta$-chain according to Definition
\ref{1Gchain} and, therefore, Proposition \ref{prop:lG} implies that  its length $M$ is bounded by
$ C_2(\sL,d) \left(2 |j'|\right)^{C_1(d)\delta}$. 
Then 
$$
\abs{j' - j} \leq \sum_{q=0}^{M-1} \abs{\Phi(j_{q+1}) - \Phi(j_q)} \leq M  \left(2 |j'|\right)^{\delta} \leq C_2(\sL,d) \left(2 |j'|\right)^{(C_1(d)+1)\delta} , 
$$
which gives
\begin{equation}
\label{jm.0}
\abs{j'} \leq \abs{j} + C_2(\sL,d) \left(2 |j'|\right)^{(C_1(d)+1)\delta} \, .
\end{equation}
As  $\delta(C_1(d)+1) < \delta_0(C_1(d)+1) \leq 1/2$, inequality \eqref{jm.0} bounds uniformly $|j'|$ in terms of $|j|$ and $C_2(\sL, d)$,  giving a contradiction since $j'$ can be chosen to have arbitrary large modulus.
\\[1mm]
\underline{ Dyadicity  of each cluster $ \Omega_\alpha $.} Denote by 
$$
m_\alpha := \min_{j \in \Omega_\alpha} |j| \, , 
\qquad
M_\alpha := \max_{j \in \Omega_\alpha} |j| \, .
$$
Let $j_\alpha$ be the index which realizes the maximum $M_\alpha = |j_\alpha|$ (it exists since each cluster is  bounded).
For any $ j , j' \in \Omega_\alpha $, let  $(j_q)_{q = 0, \ldots, L} $ be a sequence of distinct integer vectors 
in $ \Omega_\alpha $ with $ j_0 = j $, $ j_L = j' $  satisfying \eqref{rel}. Then 
$(j_q)_{q = 0, \ldots, L} $ is a $ (2 M_\alpha)^\delta $-chain according to Definition \ref{1Gchain} and,
by Proposition \ref{lG},   
\begin{equation} \label{Phijj'}
 |\Phi(j) - \Phi(j')| 
 \leq L (2 M_\alpha)^\delta 
 \leq  C_2(\sL,d)\left(2 M_\alpha \right)^{(C_1(d)+1)\delta} \, .
\end{equation}
Recalling \eqref{def:Phi}, we deduce, in particular, that
$ {\rm diam}(\Omega_\alpha) \leq C_2(\sL,d) \left(2 M_\alpha\right)^{(C_1(d)+1)\delta} $. Moreover,
since  $ \delta \leq \delta_0 $ with $ \delta_0 (C_1 (d) + 1 ) = \frac12 $,  we derive that, 
if  $ M_\alpha \geq 8 ( C_2(\sL,d))^2 =: \tC(\sL, d) $ 
then 
  \begin{equation*}
m_\alpha 
\geq 
M_\alpha -  C_2(\sL,d)\left(2 M_\alpha\right)^{(C_1(d)+1)\delta} \geq 
\frac{M_\alpha}{2} \, . 
\end{equation*}
Thus,  either $\max_{j \in \Omega_\alpha} |j| \leq \tC(\sL, d)  $  or  $\max_{j \in \Omega_\alpha} |j| \leq 2 \min_{j \in \Omega_\alpha} |j|$.
\\[1mm]
\underline{Separation properties.} 
Let $j_1, j_2 \in \Omega_\alpha$.
Consider first the case $M_\alpha \geq \tC(\sL, d) $. Then
\begin{align*}
M_\alpha \leq 2 m_\alpha \leq |j_1| + |j_2| \leq 2 M_\alpha , \quad \forall j_1, j_2 \in \Omega_\alpha \, ,
\end{align*}
which together with  \eqref{Phijj'}  imply
$$
\abs{j_1 - j_2} + \abs{ \mu_{j_1} - \mu_{j_2}} \leq 2
|\Phi(j_1) - \Phi(j_2)| \leq 
C(\sL,d) \left( |j_1| + |j_2|\right)^{(C_1(d) + 1)\delta} .
$$
If $M_\alpha \leq \tC (\sL, d) $ and $|j_1|+|j_2| \geq 1$, then 
$$
\abs{j_1 - j_2} + \abs{ \mu_{j_1} - \mu_{j_2}} \leq 2
|\Phi(j_1) - \Phi(j_2)| \leq 
C'(\sL,d) \leq C'(\sL, d) \left( |j_1| + |j_2|\right)^{(C_1(d) + 1)\delta} \, .
$$
If  $|j_1|+|j_2| < 1$, then the integers 
$j_1 = j_2 = 0$  
and \eqref{cl1} holds as well.
The proof of Theorem \ref{thm:sep}-(1) is complete.
Concerning  item $(2)$,  
 note that if $j_1 \in \Omega_\alpha$, $j_2 \in \Omega_\beta$ and $\alpha \neq \beta$, then 
 $ j_1, j_2 $ are not in the same equivalence class and therefore 
$$
\abs{j_1 - j_2} + \abs{ \mu_{j_1} - \mu_{j_2}} \geq 
|\Phi (j_1) - \Phi (j_2)| >
 \left( |j_1| + |j_2|\right)^\delta \, .
$$
The proof of Theorem \ref{thm:sep} is complete.

\section{Quasi-periodic solutions of wave equation}

As we already mentioned, we construct quasi-periodic solutions of  the nonlinear wave equation in 
\eqref{NLWa} by applying Berti-Corsi-Procesi abstract Nash-Moser theorem in \cite{BCP}, which we recall in  Appendix \ref{app:BCP}.

First define the nonlinear map
\begin{equation}\label{F:NLW}
\begin{aligned}
F(\epsilon, \lambda, \cdot) \colon H^{s+2}(\T^n \times \T^d, \R) 
& \to  H^{s}(\T^n \times \T^d, \R) \\
u & \mapsto D(\lambda)u - \e f(\varphi ,x,u)  \, , 
\end{aligned}
\end{equation}
 where   $D(\lambda)\colon  H^{s+2}(\T^n \times \T^d, \R) \to H^s(\T^n \times \T^d,\R)$ is the differential operator
\begin{equation}
\label{DW}
D(\lambda) := (\lambda \bar \omega\cdot \partial_\vf)^2 - \Delta_\sL + m \, .
\end{equation}
The set $\fK$ in \eqref{frakK} is then $\fK = \Z^n \times \Z^d \times \{1\}$, 
{that is $  \fA =\{1\} $}, and the scale of Hilbert spaces \eqref{spazi} 
are the Sobolev spaces   $H^{s}(\T^d \times \T^n, \C )$ written in Fourier variables.

\smallskip

As in the previous application, the main difficulty is to verify a property of separation of chains of singular sites of the operator $D(\lambda)$. To state it we recall some definitions from Appendix \ref{app:BCP} (see Definitions \ref{gammachainabs}--\ref{def:sigmaKn}).

Given $\Gamma \geq 2$,  a sequence of distinct integer vectors
$(\ell_q,  j_q)_{q=0,\ldots,  L} \subset \Z^n\times\Z^d$ is a  
{\em $\Gamma$-chain} if
$$
\max\{ \abs{\ell_{q+1} - \ell_q}, \,  
\abs{j_{q+1} - j_q}\} \leq \Gamma \, , \quad \forall 0\leq q \leq L-1 \, ;
$$ 
the number $L$ is called the {\em length} of the chain (see Definition \ref{gammachainabs}).

The operator  $D(\lambda)$ in \eqref{DW} is represented,  
in the basis of exponentials  $e^{\im (\ell \cdot \vf + j \cdot x)}$, by 
the infinite dimensional diagonal matrix
$$
{  {\rm diag}_{(\ell, j) \in  \Z^n \times  \Z^d} \big( D_{\ell,j} (\lambda) \big)}  \, , \quad
 D_{\ell,j}(\lambda) := - (\lambda \bar \omega \cdot \ell)^2 + \mu_j + m \, , 
$$
where $\mu_j$ are the eigenvalues of $ - \Delta_\sL $ defined in \eqref{muj}.
For $\theta \in \R$, we also define 
\begin{equation}
\label{DW2t}
D(\lambda, \theta) := {\rm diag}_{ (\ell, j) \in  \Z^n \times  \Z^d} \big( D_{\ell, j}(\lambda, \theta) \big) \, , \qquad
D_{\ell, j}(\lambda, \theta) := - (\lambda \bar \omega \cdot \ell + \theta)^2 + \mu_j + m \, .
\end{equation}
A  site $(\ell, j)\in \Z^n \times \Z^d$ is {\em singular} for $ D(\lambda, \theta)  $
if $\abs{D_{\ell,j}(\lambda, \theta)} < 1$ (see Definition  \ref{def:ssn}).

For any $\Sigma \subseteq \Z^n \times \Z^d$ and $\wt \jmath  \in \Z^d $, we denote
by $\Sigma^{\wt \jmath} $ the section of $\Sigma$ at fixed $\wt \jmath $, namely
$$
\Sigma^{\wt \jmath} := \{ (\ell,\wt \jmath) \in \Sigma \} \, . 
$$
Given  $K > 1$,   denote by $\Sigma_K$ any subset of singular sites of $D(\lambda, \theta)$ such that the cardinality of the section $\Sigma^{\wt \jmath}$ satisfies $\sharp \Sigma^{\wt \jmath}_K \leq K$,  for any $\wt \jmath \in \Sigma$ (see Definition \ref{def:sigmaKn}).

The key result is  a bound on the length of $\Gamma$-chains of singular sites of $D(\lambda, \theta)$.

\begin{proposition}[Separation of singular sites for NLW]
\label{thm:lcw}
There exist a constant $C(\sL,d, n, \tau_0)>0$ and, for any $N_0 \geq 2$, a set $\wt\Lambda = \wt\Lambda(N_0)$ such that for all 
$\lambda \in \wt\Lambda$, $\theta \in \R$ and for all $K>1, \Gamma\geq 2$ with $\Gamma K \geq N_0$, any $\Gamma$-chain of singular sites of $D(\lambda, \theta)$ 
in $\Sigma_K$ has length $L \leq (K \Gamma )^{C(\sL, d, n, \tau_0)}$.
\end{proposition}

We postpone the proof  
to  section \ref{thm3.1} and we first  conclude the proof of
Theorem \ref{main2} for NLW.

\paragraph{Verification of the assumptions of Berti-Corsi-Procesi theorem.}

First,  \eqref{D.est} is trivially satisfied with $\sigma = 2$. Moreover,   provided $f \in C^\tq$ for $\tq$ large enough and $s_0 > (d+n)/2$,  the tame estimates (f1)--(f2) hold true.

We verify now Hypothesis 1--3 {concerning the  operator 
$ D (\lambda) + \e  (\partial_u f)(\vf, x, u(\vf,x))  $
obtained linearizing \eqref{F:NLW}
at a point  $u(\vf,x) \in H^{s}(\T^d \times \T^n, \R)$. This operator has the form \eqref{L.def} with 
the multiplication operator $ T(u) := - (\partial_u f)(\vf, x, u(\vf,x)) $}. 
\\[1mm]
\underline{Verification of Hypothesis 1.}
The {diagonal} matrix $D (\lambda) = \diag \{ D_{\ell,j}(\lambda)\}$ 
{representing}  \eqref{DW} fulfills the covariance property \eqref{2.24a} choosing  $\fD_{j}(y):= -y^2 + \mu_j + m$.

The multiplication operator $ T(u) $
is represented, in  the exponential basis   $e^{\im (\ell \cdot \vf + j \cdot x)}$, by a   T\"oplitz matrix
and \eqref{2.24c}--\eqref{2.24d} hold with $\sigma_0 = 0 $.
\\[1mm]
\underline{Verification of Hypothesis 2.} {Recalling \eqref{DW2t} 
one checks that 
$$
\left\lbrace \theta \in \R \colon 
\abs{D_{\ell, j}(\lambda, \theta)} \leq N^{-\tau_1} \right\rbrace \subset I_1 \cup I_2 \, , \quad
{\rm with \  intervals}  \ I_q \  {\rm such \ that } \  {\rm meas}(I_q) \leq N^{-\tau_1} / \sqrt{m} \, . 
$$
Therefore Hypothesis 2 holds with $\mathfrak{n} = 2 ( [ 1 / \sqrt{m} \, ] + 1)  $.
}
\\[1mm]
\underline{Verification of Hypothesis 3.} It is the content of Proposition \ref{thm:lcw}.

\smallskip

The measure estimate \eqref{meas.bad1}  follows exactly as in \cite{bebo12,BCP}, and we omit it.
{Applying Theorem \ref{BCP1} we prove Theorem \ref{main2} for the NLW}.

\subsection{Proof of Proposition \ref{thm:lcw}}\label{thm3.1}

Consider the quadratic form $Q_{\sL} : \R\times \R^{d} \to \R$ defined by
\begin{equation}
\label{defQW}
Q_{\sL} (x, y) := -x^2 +  \nla{y}  \qquad {\rm where} \qquad 
 \nla{y} :=  \norm{ Wy }^2
\end{equation}
 is introduced in \eqref{def:cQ}, 
and the associated symmetric bilinear form  
$ \phi_{\sL}:  (\R \times \R^{d})^2  \to \R$, 
\begin{equation}
\label{def:phiNLW}
\phi_{\sL}\big((x,y), (x', y') \big)  := -x x' +   { \pla{y}{y'} } \, , \quad
\forall (x, y), (x',y') \in \R \times \R^{d} \, , 
\end{equation}
where 
$ \pla{y}{y'} = \la W y, W y'  \ra_{\R^d}  $ is defined in \eqref{def:phi}.

First we notice that (as in the proof of \cite[Lemma 4.2]{bebo12}) 
it is enough to bound the length of a $\Gamma$-chain
$(\ell_q, j_q)_{q=0,\ldots, L}$   of singular sites of $D(\lambda,0)$. 
Indeed, consider a $ \Gamma $-chain of singular sites $( \ell_q, j_q)_{q=0, \ldots, L} $ for $ D(\lambda, \theta )$,
 i.e.
 \begin{equation}\label{singusites}
 |(\omega \cdot \ell_q + \theta )^2 - \mu_{j_q} - m | <  1   \, , \quad \forall q = 0, \ldots, L \, , \quad 
 \omega = \lambda \bar \omega \, . 
 \end{equation}
Suppose first that $ \theta = \omega \cdot \bar \ell $ for some $ \bar \ell \in \Z^n $.
 Then the translated $ \Gamma $-chain
$( \ell_q + \bar \ell, j_q)_{q=0, \ldots, L} $ 
is formed by singular sites for $ D(\lambda, 0 )$, namely
$$
|(\omega \cdot (\ell_q + \bar \ell ))^2 - \mu_{j_q} - m | <  1   \, , \quad \forall \ell = 0, \ldots, L \, . 
$$
For any $ \theta \in \R $, we consider an approximating sequence 
$ \omega \cdot {\bar \ell}_i \to \theta $, $ {\bar \ell}_i \in \Z^n  $. 
A $ \Gamma $-chain of singular sites for $ D(\lambda, \theta)$ (see \eqref{singusites}), is, for 
$ i $ large enough, also a $ \Gamma $-chain of singular sites for $ D(\lambda, \omega \cdot {\bar \ell}_i ) $. Then
 we bound its length arguing as in the above case.

{Let $(\ell_q, j_q)_{q=0,\ldots, L}$  be a  $\Gamma$-chain of singular sites of $D(\lambda,0)$}. 
Setting
$$
x_q:= \omega \cdot \ell_q = \lambda \bar\omega \cdot \ell_q \, , \qquad \forall q = 0, \ldots, L \, ,
$$
by the definition of singular sites 
and \eqref{defQW} we have  
 \begin{equation}
\label{qw1}
\abs{Q_{\sL}(x_q, j_q) + m} < 1   \, , \qquad \forall q=0, \ldots, L \, .
\end{equation}
\begin{lemma}
\label{lem:bound_wave}
For all $ q_0, q \in \{0,\ldots, L\}$ we have
\begin{equation}
\label{bound_wave}
\abs{\phi_{\sL}\Big((x_{q_0}, j_{q_0}), (x_q - x_{q_0}, j_q - j_{q_0}) \Big)}\leq C(\sL, d,n) |q-q_0|^2 \Gamma^2 \, .
\end{equation}
\end{lemma}
\begin{proof}
By bilinearity
\begin{align*}
Q_{\sL}(x_q, j_q) = Q_{\sL}(x_{q_0}, j_{q_0})  + 2 \phi_{\sL}\big((x_{q_0}, j_{q_0}), (x_q - x_{q_0}, j_q - j_{q_0}) \big) + 
Q_{\sL}(x_q - x_{q_0}, j_q - j_{q_0}) 
\end{align*}
and then \eqref{bound_wave} follows combining \eqref{qw1} and the estimate 
\begin{align*}
\abs{Q_{\sL}(x_q - x_{q_0}, j_q - j_{q_0})} \leq |\omega \cdot (\ell_q - \ell_{q_0})|^2  + 
{ \nla{j_q - j_{q_0}} } \leq c(\sL,  d, n) \Gamma^2 |q-q_0|^2 
\end{align*}
which follows by the  definition of $\Gamma$-chain.
\end{proof}
Let us introduce  the following subspace of $\R^{d+1}$:
$$
\begin{aligned}
G &:= {\rm span }_\R \la (x_q - x_{q'}, j_q - j_{q'} )\, , \ 0 \leq q, q' \leq L \ra \\
& \ =   {\rm span }_\R \la (x_q -x_{q_0}, j_q - j_{q_0} ) \, , \ q = 0, \ldots, L \ra \,, \quad \forall q_0 =0, \ldots, L \,  . 
\end{aligned}
$$
Let $g := \dim G$. We have $ 1 \leq g \leq d + 1 $. 
We consider two cases.
\\[1mm]
{\em Case 1.}
For any $ q_0 \in \{0, \ldots, L\}$ it results that 
$
 {\rm span }_\R \la  (x_q - x_{q_0}, j_q - j_{q_0} ) \, ,  \ \ \abs{q-q_0} \leq L^\varsigma \ra = G 
$
for some $\varsigma := \varsigma (\sL, d,n, \tau_0) > 0 $ small enough, fixed later.
We select a basis of $G$ from $(x_q - x_{q_0}, j_q - j_{q_0} )$ with $\abs{q-q_0} \leq L^\varsigma $, say 
\begin{equation}
\label{fa}
f_i := (x_{q_i} - x_{q_0}, j_{q_i} - j_{q_0} ) =  \big(  \omega \cdot(\ell_{q_i} - \ell_{q_0}), j_{q_i} - j_{q_0} \big) \, ,
 \quad \forall 1 \leq i \leq g \, .
\end{equation}
Then, since $(\ell_q, j_q)_{q=0,\ldots, L}$ is a $ \Gamma $-chain,  
\begin{equation}
\label{est.fjw}
\abs{f_i} \leq n|\omega| |\ell_{q_i} - \ell_{q_0}| + |j_{q_i} - j_{q_0}| \leq n|\omega| |q_{i}-q_0| \Gamma \leq  C( d, n) L^\varsigma \Gamma 
\, , \quad \forall 1 \leq i \leq g \, .
\end{equation}
In order to derive from \eqref{bound_wave} a bound for $ (x_{q_0}, j_{q_0})$ we need a non degeneracy property for
the restriction $ ( \phi_{\sL})\vert_G $ of $ \phi_{\sL}$ to the subspace $G$. 
Then we consider the $ g \times g $  symmetric matrix 
\begin{equation}
\label{defA}
A_{\sL, \lambda} := A_{\sL, \lambda \bar \omega}  := 
(A^{i'}_{i})_{i,i'=1}^g \, , \qquad A_i^{i'} := \phi_{\sL}(f_i, f_{i'}) \, , 
\end{equation}
which represents the restriction {of the bilinear form $ \phi_{\sL} $ defined in 
\eqref{def:phiNLW}} to the subspace $ G $. 
The next key lemma, which is the technical part of our argument, shows that,  provided $\lambda$ is chosen in a set $\wt\Lambda$ of large measure,  the matrix $A_{\sL, \lambda}$ is invertible and the modulus of its  determinant  is bounded away from zero. 

\begin{lemma}
\label{lem:invW}
Let $A_{\sL, \lambda} $ 
be the matrix defined in \eqref{defA}. Then  
\begin{itemize}
\item[(i)] $\det A_{\sL, \lambda} $  has the form 
\begin{align}
\label{Pdef}
\det A_{\sL, \lambda} 
& = \norm{ \fC_g(W) p }^2 
- \lambda^2   \norm{ \fC_{g-1}(W) (\bar \omega \otimes m) }^2
\end{align}
where $ \fC_g(W)$ is the $g-th$ compound matrix of $W$, 
$\norm{\cdot}$  the euclidean norm,  $ p  \in  \Z^{\binom{d}{g}} $ and  
$ \bar \omega \otimes m := ( \bar \omega \cdot  m_{\vec a})_{\vec a} \in \R^{\binom{d}{g-1} n}$,  
$m_{\vec a} \in \Z^n $,  $ \vec a = (a_1,  \ldots , a_{g-1}) $ with  
$ 1 \leq a_1 < \ldots < a_{g-1} \leq d $. 
Each vector $m_{\vec a} $ satisfies  
\begin{equation}
\label{p.m.est}
|m_{\vec a} | \leq C( d)(\Gamma L^\varsigma)^{g} \,. 
\end{equation}
\item[(ii)]
 $\det A_{\sL, \lambda} $ is not  identically zero as function of  $ \lambda^2 $.
\item[(iii)] For any $N_0$ sufficiently large, there   exists a set $\wt\Lambda := \wt\Lambda(N_0) \subset \Lambda$ with $\meas(\Lambda \setminus \wt\Lambda) \leq O(N_0^{-1})$ such that  for any $\lambda \in \wt\Lambda$ one has,
for some $ c(d) > 0 $ and $\tau_2 := \tau_2 (d,n, \tau_0) $, 
\begin{equation}
\label{AW}
\abs{\det A_{\sL, \lambda}}  \geq \frac{c(d)}{N_0   (\Gamma L^\varsigma)^{\tau_2}}  \, . 
\end{equation}
\end{itemize}
\end{lemma}

\noindent
We postpone the proof of Lemma \ref{lem:invW} to  section \ref{sec:5.3}, first concluding 
the proof of Proposition \ref{thm:lcw}.

Take $\lambda \in \wt\Lambda$, so that, by the previous lemma,   the matrix 
$ A_{\sL, \lambda}$ is invertible.
Therefore  the symmetric bilinear form $\phi_{\sL}\vert_G$ is nondegenerate and it induces the splitting 
$$
\R^{d+1} = G \oplus G^{\bot \phi_{\sL}} \qquad {\rm where} \qquad
G^{\bot \phi_{\sL}} := \left\lbrace
z \in \R^{d+1} \colon \phi_{\sL}(z, f) = 0  \quad \forall f \in G  \right\rbrace \, . 
$$
We denote by $P_G :\R^{d+1} \to G$ the  corresponding projector on $G$.

In order to  obtain bounds for  the projection of each $P_G (x_{q_0}, j_{q_0}) $, 
we decompose it on the basis $(f_i)_{1 \leq i \leq g}$, 
\begin{equation}\label{PG:dec}
P_G (x_{q_0}, j_{q_0}) = \sum_{i=1}^g z_i f_i \, ,  
\end{equation}
and we look for  bounds of the coordinates $z=(z_i)_{1 \leq i \leq g}$.
The $z$ are determined by solving  the linear system  
$A_{\sL, \lambda} z = \bb$, where  $A_{\sL, \lambda} $ is the matrix in \eqref{defA} and 
$$
\bb := (\bb_l)_{1 \leq l \leq g} \in \R^g \, , \qquad
\bb_l := \phi_{\sL}( P_G (x_{q_0}, j_{q_0}), f_l) = \phi_{\sL}(  (x_{q_0}, j_{q_0}), f_l) \, . 
$$
Then the Cramer rule, the estimates \eqref{bound_wave}, \eqref{AW} and 
{$ |A_l^{i}| \leq c(\sL, d,n) |f_i| |f_{l}| \leq C(\sL,d,n) (\Gamma L^\varsigma)^2$, by \eqref{est.fjw}},  give
\begin{equation}
\label{normz}
\abs{z_i}  \leq {N_0} \, C(\sL,d,n) (\Gamma L^\varsigma)^{2g + \tau_2} , \qquad \forall 1 \leq i \leq g \, . 
\end{equation}
From  {\eqref{PG:dec}},   \eqref{normz}, \eqref{est.fjw},  we deduce 
$$
|P_G (x_{q_0}, j_{q_0})| \leq {N_0} \, C(\sL,d,n) 
(\Gamma L^\varsigma)^{2g + \tau_2+1} \, , \quad \forall q_0 \in \{0, \ldots, L \} \, .
$$
Therefore we get that, for all $  q_1, q_2 \in \{0, \ldots, L \} $,  
$$
|(x_{q_1}, j_{q_1}) - (x_{q_2}, j_{q_2})| = 
|P_G (x_{q_1}, j_{q_1}) - P_G (x_{q_2}, j_{q_2})|  \leq 
{N_0} \, C(\sL,d,n) (\Gamma L^\varsigma)^{2g + \tau_2+1} \, .
$$
In particular, for all  $ q_1, q_2 \in \{0, \ldots, L \} $ we have also
$|j_{q_1} - j_{q_2}|  \leq {N_0} \, 
 C(\sL,d,n) (L^\varsigma \Gamma)^{2d + \tau_2+1}$, and so the
 sequence  $\{j_q\}_{q = 0, \ldots , L}$ is  contained in a ball of diameter 
$ N_0 \, C(\sL,d,n) (L^\varsigma \Gamma)^{2d + \tau_2+1}$. 
 Since they are integer vectors, their number (counted without multiplicity), does not exceed 
 $ {N_0^d}  C(\sL,d,n) (L^\varsigma \Gamma)^{(2d + \tau_2+1)d}$.
By the assumptions of Proposition \ref{thm:lcw} the 
{$ \Gamma $-chain 
 $(\ell_q, j_q)_{q=0,\ldots, L}$  is in $ \Sigma_{K} $, thus, 
for any $q_0\in \{0, \ldots,  L \}$, the cardinality of the set  of singular sites of the chain with fixed $j_{q_0}$ is bounded by} 
 $$
\sharp{ \{ (\ell_q, j_q)_{q = 0, \ldots ,L} \colon j_q = j_{q_0} \} } \leq K \, ,
 $$
and we deduce  that 
 $$
 L \leq {N_0^d} \, C(\sL,d,n) (L^\varsigma \Gamma)^{(2d + \tau_2+1)d} K \leq
 {(\Gamma K)^d} \, C(\sL,d,n) (L^\varsigma \Gamma)^{(2d + \tau_2+1)d} 
 \, .
 $$
 By choosing $\varsigma $ 
such that $ \varsigma (2d + \tau_2+1)d \leq 1 / 2 $,  we deduce that $L \leq 
C(\sL,d,n)  {(K \Gamma)^{C(d, \tau_2)}} $.
 \\[1mm]
{\em Case 2.}  $\exists q_0 \in \{ 0, \ldots, L \} $ for which
$$
G_1 := {\rm span }_\R \la  (x_q - x_{q_0}, j_q - j_{q_0} ) \, ,  \ \ \abs{q-q_0} \leq L^\varsigma \ra \subsetneq G \, , 
\quad  g_1 := \dim G_1 \leq g-1 \, ,
$$
so all vectors $(x_q - x_{q_0}, j_q - j_{q_0} )$ belong to an affine subspace of dimension $1 \leq g_1 \leq g-1$. Then we repeat the argument of Case 1 on the subchain $\{ (\ell_q, j_q) \colon \ |q-{q_0}| \leq L^\varsigma \}$, obtaining 
a bound on $L^\varsigma $, thus for $ L $.
Applying this procedure at most $d+1$ times, we obtain a bound on $L$ of the claimed form 
$L \leq (K \Gamma)^{C(\sL, d, n, \tau_0)}$.

\subsection{Proof of Lemma \ref{lem:invW}}\label{sec:5.3}

Write the bilinear form $\phi_{\sL}$ in \eqref{def:phiNLW} as the sum of two symmetric bilinear forms:
\begin{equation}
\label{split}
\begin{aligned}
& \phi_{\sL}  = - \phi_1 + \phi_2 \, , \\
&\phi_1\big((x,y), (x', y') \big)  :=  x x' , 
\qquad
\phi_2\big((x,y), (x', y') \big)  := {\pla{y}{y'}} \, , 
\end{aligned}
\end{equation}
where
$ {\pla{y}{y'}} =  \la W y, W y'  \ra_{\R^d} $ and  
$ W = \lV^{- \top}$ {is defined in \eqref{VeW}}.

Correspondingly  we decompose the symmetric matrix $A_{\sL,\lambda}  $ in \eqref{defA}, 
as $ A_{\sL,\lambda} = - S + R $
where the symmetric matrices 
$S = (S_i^{i'})_{1 \leq i, i' \leq g}$ and $R = (R_i^{i'})_{1 \leq i, i' \leq g}$ are given by
\begin{align}
\label{defS1}
& 
S_{i}^{i'} :=  \phi_1(f_{i'}, f_{i}) = (\omega \cdot  l_{i'}) \, (\omega \cdot l_{i}) \, , 
\qquad l_i:= \ell_{q_i} - \ell_{q_0} \in \Z^n \, , \\
& \label{defR1}
R_i^{i'} := \phi_2(f_{i'}, f_{i}) = { \pla{k_{i'}}{ k_{i}} }  \,  , 
\qquad \quad  \ \ k_i:= j_{q_i} - j_{q_0} \in \Z^d \, . 
\end{align}
Since $(\ell_q, j_q) $, with indices $ |q - q_0 | \leq L^\varsigma $,  is a $\Gamma$-chain, we have
\begin{equation}
 \label{liki1}
 \max\left\lbrace{|l_i|, |k_i|}\right\rbrace 
\leq |q_i - q_0| \Gamma \leq L^\varsigma \Gamma \ , \quad \forall 1 \leq i \leq g \, .  
 \end{equation} 
We write $ S = ( S_1 | \cdots | S_g ) $ and $ R = (R_1| \cdots | R_g ) $ where $ S_i $, $ R_i \in \R^g $,
$ i = 1, \ldots, g $ denote the columns of $ S $ and  $ R $.
The matrix $ S $ has rank 1 since all its columns $S_i \in \R^g$ are colinear:
\begin{equation}\label{def:ell1}
S_i = (\omega \cdot l_i ) \, \bl \,  , \qquad \forall i = 1, \ldots , g \,
 \, ,  \qquad \mbox{ where }  \ \ \bl:= \begin{pmatrix}
\omega \cdot l_1 \\
\vdots\\
\omega \cdot l_g
\end{pmatrix} \,  . 
\end{equation}
By the multilinearity of the determinant along each column, 
we develop  $ \det A_{\sL, \lambda} $ as
\begin{equation}\label{detASR1}
\det A_{\sL, \lambda} = \det (-S + R) 
 = \det R - { \det(S_1| R_2| \cdots | R_g) - \ldots  - \det(R_1| \cdots| R_{g-1}| S_g)} 
\end{equation}
where we used that the determinant of a matrix with two colinear columns $S_i, S_{i'}$ is null.

Now remark that $\det R$ is the same that we computed in Lemma \ref{detA} {(substituting
 $ k_i \rightsquigarrow f_i $ in \eqref{l.eqA})}, so we obtain 
\begin{align}
\label{detR}
\det R & =  \norm{\fC_g(W)p}^2
\end{align}
for some vector  $ p \in \Z^{\binom{d}{g}}$ {(notice that 
$ p $ could be zero since the vectors
$ k_1, \ldots, k_g $ may not be independent).}

In the next lemma we  compute  $\det(S_1| R_2| \cdots |R_g) + \cdots  + \det(R_1| \cdots | R_{g-1} | S_g)$.
\begin{lemma}\label{lem:A.w}
One has
\begin{equation}
\label{det2w}
\begin{aligned}
{\det(S_1 | R_2 | \cdots | R_g)}  & + \cdots  + { \det(R_1 | \cdots| R_{g-1} | S_g)}  \\
&=  \lambda^2  \sum_{1 \leq c_1 < \cdots < c_{g - 1} \leq d}
\left( \sum_{1 \leq a_1 < \ldots < a_{g-1}\leq d} 
\det(W_{c_1 \cdots c_{g-1}}^{a_1 \cdots a_{g-1}}) \, \bar \omega \cdot m_{a_1 \ldots a_{g-1}} \right)^2
\end{aligned}
\end{equation}
with integer coefficients $m_{a_1 \ldots a_{g-1}} \in \Z^n$  satisfying
\begin{equation}
\label{m.wave}
|m_{a_1 \ldots a_{g-1}}| \leq C(d) (L^\varsigma \Gamma)^{g} \, . 
\end{equation}
\end{lemma}

\begin{proof}
Developing the matrix $R$ in \eqref{defR1}  (recall \eqref{def:phi} and \eqref{VeW}) 
we get
$$
R = \sum_{a,b,c=1}^d W_c^a \,  W_c^b \, R^{(ab)}_s \, , 
\qquad R^{(ab)} := \begin{pmatrix}
k_1^{(a)} \\ 
\vdots\\
k_g^{(a)}
\end{pmatrix} 
\begin{pmatrix}
k_1^{(b)} ,  
\ldots , k_g^{(b)}
\end{pmatrix} \, , 
$$
where $k_i^{(a)}$ is the $a^{th}$-component of the vector $k_i$.
Note that the $g \times g$ {symmetric} matrix $R^{(ab)}$ has colinear columns, with the $s$ column given by
\begin{equation}\label{def:vco1}
R^{(ab)}_s =  k_s^{(b)}  \bkn^{(a)}  , \qquad \bkn^{(a)} := \begin{pmatrix}
k_1^{(a)} \\ 
\vdots\\
k_g^{(a)}
\end{pmatrix} \, , \qquad a,b=1, \ldots, d \, , \quad s=1, \ldots, g \, . 
\end{equation}
Writing the columns of $R$ as 
$R_s = \sum_{a,b,c=1}^d W_c^a \,  W_c^b  \, R_s^{(ab)}$ and substituting in the first line of \eqref{det2w},  we get,
by the multilinearity of the determinant, 
\begin{align*}
& \det(S_1 | R_2 |  \cdots | R_g) + \cdots  + \det(R_1| \cdots |R_{g-1} | S_g) \\
& =  \sum_{m_2, \ldots, m_g=1 \atop { n_2, \ldots, n_g =1 \atop o_2, \ldots, o_g = 1}}^d  
\ \ \prod_{2 \leq h \leq g}W_{o_h}^{m_h} W_{o_h}^{n_h}\cdot
\det \big(S_1 | R^{(m_2 n_2)}_2 | \cdots | R^{(m_g n_g)}_g \big) \\
& \qquad  + \ldots + 
 \sum_{m_1, \ldots, m_{g-1}=1 \atop { n_1, \ldots, n_{g-1} =1 \atop o_1, \ldots, o_{g-1} = 1}}^d 
\ \ \prod_{1 \leq h \leq g-1}W_{o_h}^{m_h} W_{o_h}^{n_h}\cdot 
\det \big( R^{(m_1 n_1)}_1 |  \cdots | R^{(m_{g-1} n_{g-1})}_{g-1} | S_g \big) \\
&  \stackrel{ {\eqref{def:ell1}, \eqref{def:vco1}} } =
 \sum_{m_1, \ldots, m_{g-1}=1 \atop { n_1, \ldots, n_{g-1} =1 \atop o_1, \ldots, o_{g-1} = 1}}^d 
\ \ \prod_{1 \leq h \leq g-1}W_{o_h}^{m_h} W_{o_h}^{n_h}\cdot 
  (\omega \cdot l_1) k_2^{(n_1)} \cdots k_g^{(n_{g-1})} 
\det \big(\bl | \bkn^{(m_1)} |  \cdots | \bkn^{(m_{g-1})} \big) 
 \\
&   + \ldots +  \sum_{m_1, \ldots, m_{g-1}=1 \!\!\!\! \!\!\!\! \atop { n_1, \ldots, n_{g-1} =1 \atop o_1, \ldots, o_{g-1} = 1}}^d
\  \prod_{1 \leq h \leq g-1}W_{o_h}^{m_h} W_{o_h}^{n_h}
  (-1)^{g-1}  (\omega \cdot l_g) k_1^{(n_1)} \cdots k_{g-1}^{(n_{g-1})}  \det \big(\bl |  \bkn^{(m_1)} | \cdots | \bkn^{(m_{g-1})} \big) \, .
\end{align*}
Note that, in the last equality, we have renamed the indices  in the first sum and   reordered the columns of the matrix inside the determinant to have the vector $\bl$ always in the first position.
Now we notice that 
the only non {zero} determinants are those where the indices $m_i$ are all different, otherwise at least two columns are colinear. 
Thus the sums above are on distinct indices, namely in each sum  $m_i \neq m_{i'}$ for $i \neq i'$.
Therefore the 
$m_i$'s in each monomial are just a permutation of some $1 \leq a_1  < \cdots< a_{g-1} \leq d$. Denote by $\sigma_{\vec a}$ the permutation of these $ g - 1 $ indices defined by
$\sigma_{\vec a}(a_i) = m_i$ $\ \forall i$, and
by  $ {\rm sign }\, \sigma_{\vec a} $ its signature. 
By the alternating property of the determinant, we have
$$
\det \big(\bl | \bkn^{(m_1)} | \cdots | \bkn^{(m_{g-1})} \big)  = 
\det \big(\bl | \bkn^{(a_1)} | \cdots | \bkn^{(a_{g-1})} \big) \, {\rm sign }\, \sigma_{\vec a} \, . 
$$
Therefore $\det(S_1 | R_2 | \cdots |  R_g) + \ldots  + \det(R_1 |  \cdots |  R_{g-1} | S_g)$ equals 
\begin{equation}
\label{derW1}
\begin{aligned}
& \sum_{1 \leq a_1 < \cdots < a_{g - 1} \leq d} 
 \det \big(\bl |  \bkn^{(a_1)} | \cdots | \bkn^{(a_{g-1})} \big)  \sum_{n_1, \ldots, n_{g-1} = 1 \atop o_1, \ldots, o_{g-1} = 1}^d 
 \left(\sum_{\sigma_{\vec a} \in S_{g-1}}
{\rm sign }\,  \sigma_{\vec a} 
\prod_{1 \leq h \leq g-1}W_{o_h}^{\sigma_{\vec a}(a_h)}  \right)\\
& \quad \quad \times \prod_{1 \leq h \leq g-1}W_{o_h}^{n_h} \Big[
 (\omega \cdot l_1)k_2^{(n_1)} \ldots k_{g}^{(n_{g-1})}   +  \ldots + (-1)^{g-1} (\omega \cdot l_g)  k_1^{(n_1)}\ldots  k_{g-1}^{(n_{g-1})} 
\Big] 
\end{aligned}
\end{equation}
where $ S_{g-1} $ denotes the group of all possible permutations of the indices $ a_1 < \cdots < a_{g-1}$. Now one has
$$
\sum_{\sigma_{\vec a} \in S_{g-1}}
{\rm sign }\,  \sigma_{\vec a} \, W_{o_1}^{\sigma_{\vec a}(a_1)} \cdots W_{o_{g-1}}^{\sigma_{\vec a} (a_{g-1})}   = \det (W_{o_1 \ldots o_{g-1}}^{a_1 \ldots a_{g-1}})
$$
and this determinant is not zero only if the $o_i$ are permutation of some indices $1 \leq c_1 < \ldots < c_{g-1} \leq d$. As above, we denote by $\sigma_{\vec c} $ 
the permutation such that  $\sigma_{\vec c} (c_i) = o_i $, for all $ i $.
We obtain
\begin{equation}
\label{derW11}
\begin{aligned}
\eqref{derW1} = \!\!\!\!\!\!
 & \sum_{1 \leq a_1 < \cdots < a_{g - 1} \leq d \atop
1 \leq  c_1 < \ldots < c_{g-1} \leq d}
 \det \big(\bl |  \bkn^{(a_1)} | \cdots | \bkn^{(a_{g-1})} \big) \det (W_{c_1 \ldots c_{g-1}}^{a_1 \ldots a_{g-1}})\\
& \times \sum_{n_1, \ldots, n_{g-1}=1}^d \Big[
 (\omega \cdot l_1)k_2^{(n_1)} \ldots k_{g}^{(n_{g-1})}   +  \ldots + (-1)^{g-1} (\omega \cdot l_g)  k_1^{(n_1)}\ldots  k_{g-1}^{(n_{g-1})} \Big] \\
 & \times  \sum_{\sigma_{\vec c} \in S_{g-1}} 
{\rm sign} \, \sigma_{\vec c} \, 
 W_{c_1}^{n_1} \ldots W_{c_{g-1}}^{n_{g-1}} \, .
\end{aligned}
\end{equation}
Again, the last line equals
$$
\sum_{\sigma_{\vec a} \in S_{g-1}}
{\rm sign }\,  \sigma_{\vec c} \, W_{c_1}^{n_1}\ldots W_{c_{g-1}}^{n_{g-1}}   = \det (W_{c_1 \ldots c_{g-1}}^{n_1 \ldots n_{g-1}})
$$
and the determinant is not zero only if the $n_i$ are permutation  $ \sigma_{\vec b} $ 
of some indices $1 \leq b_1 < \ldots < b_{g-1} \leq d$.
We obtain thus
\begin{equation}
\label{derW12}
\begin{aligned}
\eqref{derW11} =  \!\!
 & \sum_{1 \leq a_1 < \cdots < a_{g - 1} \leq d \atop
{1 \leq  c_1 < \ldots < c_{g-1} \leq d \atop 1 \leq b_1 < \ldots < b_{g-1} \leq d}}
 \det \big(\bl |  \bkn^{(a_1)} | \cdots | \bkn^{(a_{g-1})} \big) \det (W_{c_1 \ldots c_{g-1}}^{a_1 \ldots a_{g-1}})
\, 
\det (W_{c_1 \ldots c_{g-1}}^{b_1 \ldots b_{g-1}}) 
 \\
& \times \sum_{\sigma_{\vec b} \in S_{g-1}} \Big[
 (\omega \cdot l_1)k_2^{\sigma_{\vec b} (b_1)} \ldots k_{g}^{\sigma_{\vec b} (b_{g-1})}   +  \ldots + (-1)^{g-1} (\omega \cdot l_g)  k_1^{\sigma_{\vec b} (b_1)}\ldots  k_{g-1}^{\sigma_{\vec b} (b_{g-1})} \Big] \,  .
\end{aligned}
\end{equation}
Now, the expression in the last  line of \eqref{derW12} is nothing but the determinant of the matrix 
\begin{equation}\label{det-plus1}
\begin{pmatrix}
\omega \cdot l_1   &  k_1^{(b_1)} & \ldots &  k_1^{(b_{g-1})}   \\
\omega \cdot l_2   &  k_2^{(b_1)} & \ldots &  k_2^{(b_{g-1})}   \\
\vdots & \vdots & \ldots & \vdots   \\
\omega \cdot l_g & k_g^{(b_1)} & \ldots & k_g^{(b_{g-1})}
\end{pmatrix} \stackrel{ {\eqref{def:ell1}, \eqref{def:vco1}} } \equiv  \begin{pmatrix}
\bl |  \bkn^{(b_1)}|  \cdots | \bkn^{(b_{g-1})} 
\end{pmatrix}  
\end{equation}
as one sees easily by {a Laplace expansion of the determinant along the first column and recalling the} 
Leibnitz formula for the determinant. 
In conclusion we have proved that 
\begin{align*}
\eqref{derW12} & = 
\sum_{1 \leq a_1 < \cdots < a_{g - 1} \leq d \atop
{1 \leq  c_1 < \ldots < c_{g-1} \leq d \atop 1 \leq b_1 < \ldots < b_{g-1} \leq d}} 
 \det \big(\bl |  \bkn^{(a_1)} | \cdots | \bkn^{(a_{g-1})} \big) 
 \det (W_{c_1 \ldots c_{g-1}}^{a_1 \ldots a_{g-1}}) \\
& \qquad \qquad \qquad \ \ \times \det (W_{c_1 \ldots c_{g-1}}^{b_1 \ldots b_{g-1}}) \det \big(\bl |  \bkn^{(b_1)} | \cdots | \bkn^{(b_{g-1})} \big) \\
& = \sum_{1 \leq  c_1 < \ldots < c_{g-1} \leq d} 
\left(\sum_{ 1 \leq a_1 < \ldots < a_{g-1}\leq d}
\det (W_{c_1 \ldots c_{g-1}}^{a_1 \ldots a_{g-1}}) \, \det \big(\bl |  \bkn^{(a_1)} | \cdots | \bkn^{(a_{g-1})} \big)
 \right)^2 \, . 
\end{align*}
Now, \eqref{det2w}-\eqref{m.wave} follow  since 
$$
\det \big(\bl |  \bkn^{(a_1)} | \cdots | \bkn^{(a_{g-1})} \big) = \omega \cdot m_{a_1 \ldots a_{g-1}} = \lambda \bar \omega \cdot m_{a_1 \ldots a_{g-1}} 
$$
where $m_{a_1 \ldots a_{g-1}}$ is a vector in $\Z^n$ 
satisfying 
$$
\abs{m_{a_1 \ldots a_{g-1}}} \leq C(d) \max_{1 \leq i \leq g} |l_i|\, \max_{1 \leq i \leq g} |k_i|^{g-1} \leq 
C(d) (L^\varsigma \Gamma)^g  
$$
by \eqref{liki1}. 
\end{proof}

Formula  \eqref{detR}, Lemma \ref{lem:A.w}  and  \eqref{detASR1} imply \eqref{Pdef}.

\begin{lemma}\label{lem:setcG}
Consider  
$$
P_{p,m}(\lambda^2) 
:=  \norm{ \fC_g(W)  p }^2 
 - \lambda^2 \norm{ \fC_{g-1}(W)  (\bar \omega \otimes  m)  }^2
$$
where {$ p  \in \Z^{\binom{d}{g}} $} 
and 
$ m := (m_{\vec a}) \in \Z^{\binom{d}{g-1} n}$, $ \vec a = (a_1, \ldots, a_{g-1} ) $
with $ 1 \leq a_1 < \ldots < a_d \leq g $, and  
$ \bar \omega \otimes m 
:= ( \bar \omega \cdot  m_{\vec a})_{\vec a} \in \R^{\binom{d}{g-1} n}$.
Let $ |m|:= \max |m_{\vec a}| $. 
Assume that  $\bar \omega \in \R^n $ is Diophantine, according to  \eqref{omega1}. 
Then there exists $ \mathfrak c(\sL)>0$ such that  for $ \tau \geq \tau_1 (d,n, \tau_0) $ large enough, 
for any  $\gamma \in \big[0, \frac{1}{4} \mathfrak c(\sL) \big] $,
the set
{\begin{equation*}
\wt\Lambda := \left\lbrace 
\lambda \in \Lambda \,  
\colon \, 
|P_{p,m}(\lambda^2)| \geq \frac{\gamma}{1+ |m|^{\tau}} \, , \, \forall   (p,m) \neq (0,0) 
\right\rbrace
\end{equation*}}
has large measure, more precisely 
\begin{equation}
\label{measG}
\meas (\Lambda \setminus \wt\Lambda) \leq {C(\sL, \bar \omega, d, \gamma_0)} \gamma  \ .
\end{equation}
\end{lemma}
\begin{proof}
For $ p \in {\Z^{\binom{d}{g}}} $ and $ m  \in \Z^{\binom{d}{g-1} n}$, let 
$$
\eta_p:= \norm{  \fC_g(W)  p }^2 \, , \quad
\zeta_m(\bar \omega):=  \norm{ \fC_{g-1}(W)  (\bar \omega \otimes  m) }^2  \, , 
\quad  P_{p,m}(\xi) = \eta_p - \xi \zeta_m(\bar \omega) \, , \quad  \xi := \lambda^2 \, . 
$$
We have that  $\meas (\Lambda \setminus \wt\Lambda) \leq {C} \sum_{(p,m) \neq (0,0)} \meas (\cR_{p,m})$, where 
$$
\cR_{p,m} := \left\lbrace
\xi = \lambda^2 \in \left[ \frac{1}{4}, \frac94 \right] \ \colon  \ \ |P_{p,m}(\xi)| < \frac{\gamma }{1+ |m|^{\tau}}
\right\rbrace \, .  
$$
We distinguish two cases.
\\[1mm]
{\bf Case 1:} $p \neq 0$.
Therefore
$$
\eta_p  = \norm{ \fC_g(W) p }^2 \geq \frac{1}{\norm{\fC_g(W)^{-1}}^2} \norm{p}^2 
\geq \mathfrak c(\sL)  |p|^2  \geq \mathfrak c(\sL) \, .
$$
 If $\cR_{p,m} \neq \emptyset$ then, since 
 $|\xi| \leq 9/4$ and $\gamma \in \big[0, \frac{1}{4} \mathfrak c(\sL) \big]$, we deduce that 
{$$
|\zeta_m(\bar \omega)| \geq  \frac{\mathfrak c(\sL) }{3},  \ \ 
 |p|^2\leq 3 \frac{|\zeta_m( \bar \omega)|}{ \mathfrak c(\sL)} \leq C(\sL, \bar \omega, d)|m|^2
  \, , \ \
\meas (\cR_{p,m}) \leq  \frac{2\gamma }{1+ |m|^{\tau}} \frac{1}{|\zeta_m(\bar \omega)|} \, . 
$$}
Hence, for $ \tau := \tau (d,n) $ sufficiently large, 
\begin{equation} \label{measRpm}
\sum_{p \neq 0 , \,	m} \meas (\cR_{p,m})  = 
\sum_{0 < |p|\leq C |m| , \,  m } \meas (\cR_{p,m})  
\leq  C(\sL, \bar \omega, d, n) \gamma  \, . 
\end{equation}
{\bf Case 2:} $ p = 0 $. 
In this case $m \neq 0$ and, by the invertibility of the compound 
matrix $ \fC_{g-1}(W) $ and the diophantine condition \eqref{omega1}, we get
$$
\zeta_m(\bar \omega) \geq \mathfrak{c}_{g-1}(\sL)
\norm{\bar \omega \otimes m}^2 \geq 
\mathfrak{c}_{g-1}(\sL)  \abs{\bar \omega \otimes m}^2 \geq \mathfrak{c}_{g-1}(\sL) 
 \frac{\gamma_0^2}{|m|^{2\tau_0}} \, .
$$
 We deduce that,  
for $\tau \geq 2\tau_0 $, 
$$
\cR_{0, m} \subseteq \left( 0, \frac{\gamma}{1+|m|^{\tau}} \, \frac{| m |^{2 \tau_0}}{\gamma_0^2 \, \mathfrak c(\sL)} \right]
\subseteq \left(0,  \frac{\gamma}{\gamma_0^2 \, \mathfrak c(\sL)}\right] \,. 
$$
This inclusion and \eqref{measRpm} prove \eqref{measG}.
\end{proof}

\begin{proof}[Proof of Lemma \ref{lem:invW} $(i)$--$(iii)$]
Item $(i)$  follows combining  \eqref{detASR1} with \eqref{detR} and  Lemma \ref{lem:A.w}.
To prove item $(ii)$ it is sufficient to notice that $\det A_{\sL, \lambda}$ is not zero when evaluated at $\im \lambda$.
 Indeed the matrix $A_{\sL, \im \lambda}$ is the Gram matrix of the vectors $f_i$ with respect to the scalar product $\phi_1 + \phi_2$. 
 Being $f_1, \ldots, f_g$ linearly independent, 
 $\det A_{\sL, \im \lambda}>0 $. {In particular the integer vectors $ (p, m) \neq (0,0) $.}
Finally item $(iii)$  follows by Lemma \ref{lem:setcG} and the bound  \eqref{p.m.est} for $ |m_{\vec a}|$,  
choosing $ N_0 = \gamma^{-1} $  
and $ \tau_ 2 := g \tau_1  (d, n, \tau_0 )  $. 
\end{proof}

\section{Quasi-periodic solutions of Schr\"odinger equation}

It is convenient to consider the nonlinear Schr\"odinger equation in \eqref{NLWa} coupled with its complex conjugate equation, so we look for zeros of the  nonlinear operator 
\begin{equation}
\begin{aligned}
\label{NLSavector}
F(\epsilon, \lambda, \cdot) \colon H^{s+2}(\T^{n+d},\C) \times H^{s+2}(\T^{n+d}, \C) & \to
H^{s}(\T^{n+d},\C) \times H^{s}(\T^{n+d}, \C)
 \\
\begin{pmatrix} u^+ \\ u^- \end{pmatrix}  & \mapsto 
 D(\lambda)\begin{pmatrix} u^+ \\ u^- \end{pmatrix} - \epsilon f(u^+, u^-) \, , 
\end{aligned}
\end{equation}
where  $u^\pm$ are functions of the periodic variables $(\vf, x)\in \T^n \times \T^d$, $D(\lambda)$ is the 
differential  operator
 \begin{equation*}
D(\lambda) := \begin{pmatrix}
\im \lambda \bar \omega  \cdot \partial_\vf - \Delta_\sL + m & 0 \\
0 &
-\im \lambda \bar \omega \cdot \partial_\vf - \Delta_\sL + m
\end{pmatrix}  
\end{equation*}
and $f(u^+, u^-)$ is the nonlinearity
\begin{equation*}
 f(u^+, u^-) := \begin{pmatrix}
 \tf^+(\varphi, x, u^+, u^-)  \\ 
\tf^-(\varphi, x, u^+,u^-) 
\end{pmatrix} ;
\end{equation*}
here 
 $\tf^\pm( u, v)$  are two extensions (in the real sense) of $\tf(\vf, x, u)$ so that $\tf^+( u, \bar u) = \bar{\tf^-( u, \bar u)} = \tf(u)$ and $\partial_u \tf^+(u, \bar u) = \partial_v \tf^-(u, \bar u) \in \R$, 
 $
 \partial_{\bar u} \tf^+(u, \bar u) = \partial_{\bar u} \tf^-(u, \bar u) = \partial_{\bar v} \tf^+(u, \bar u) = 
 \partial_{\bar v} \tf^-(u, \bar u) = 0$ and
 $\partial_v \tf^+(u, \bar u) = \bar{\partial_u \tf^-(u, \bar u)} $, see e.g. \cite{BCP}. 

We look for zeros of $F$ in the subspace
$$
\cU := \left\lbrace  u =(u^+, u^-) \in H^s(\T^{n+d},\C) \times H^s(\T^{n+d},\C)  \ \ \colon \ \  u^+ = \bar{u^-} \right\rbrace \, .
$$ 
The set $\fK$ of \eqref{frakK} is then $\fK = \Z^n \times \Z^d \times \{1,-1\}$ and the scale of Hilbert spaces
\eqref{spazi} is  $H^{s}(\T^{n+d},\C) \times H^{s}(\T^{n+d}, \C)$ written in Fourier variables.

\smallskip
As in the application to NLW, the main difficulty is to verify a bound 
for the length of chains of singular sites of  $D(\lambda)$. We begin by
{specializing} Definitions \ref{gammachainabs}--\ref{def:sigmaKn} to the case of NLS.

First,  given $\Gamma \geq 2$, 
a sequence of distinct integer vectors 
$(\ell_q,  j_q, \fa_q)_{q=0,\ldots,  L} \subset \Z^n\times\Z^d\times\{1, -1\}$ is  a  
{\em $\Gamma$-chain} if
{ $ \max\{ \abs{\ell_{q+1} - \ell_q}, \,  
\abs{j_{q+1} - j_q}, \abs{\fa_{q+1} - \fa_q} \} \leq \Gamma $, $ \forall 0\leq q \leq L-1 $}. 
The number $L$ is called the {\em length} of the chain (see Definition \ref{gammachainabs}).

The operator 
 $D(\lambda)$ in the basis of exponentials  
\begin{equation}
\label{basisNLS}
e^{\im \ell \cdot \vf} \be_{j,\fa}(x) \, , \qquad 
\be_{j,\fa}(x):= \left\lbrace\begin{matrix}
(e^{\im j \cdot x}, 0) & \mbox{ if } \fa = +1 \\
( 0, e^{\im j \cdot x}) &  \mbox{ if } \fa = -1 
\end{matrix} \, , 
\right. 
\end{equation}
  is represented by the infinite dimensional matrix 
 $$
D(\lambda) := {\rm diag}_{{\ell,j, \fa}} \left(D_{\ell, j, \fa}(\lambda) \right) , \quad
D_{\ell, j, \fa}(\lambda) := -\fa \lambda \bar \omega \cdot \ell  + \mu_j + m \,  , \quad 
(\ell, j, \fa) \in \Z^n \times \Z^d \times \{1, -1\} \, , 
 $$
 where $\mu_j$ are defined in \eqref{muj}.
 For $\theta \in \R$ define also
 $$
D(\lambda, \theta) := {\rm diag} \left(D_{\ell, j, \fa}(\lambda, \theta) \right) , \qquad
D_{\ell, j, \fa}(\lambda, \theta) := -\fa (\lambda \bar \omega \cdot \ell + \theta)  + \mu_j + m \, .
 $$
Now recall that a site    $(\ell,  j , \fa) \in \Z^n \times \Z^d\times \{-1, 1\}$ is called {\em singular } if 
$\abs{D_{\ell,j, \fa}(\lambda, \theta)} < 1$ (see Definition \ref{def:ssn}). \\
For any $\Sigma \subseteq \Z^n \times \Z^d \times \{1,-1\}$ and $\wt \jmath \in \Z^d $, we denote
by $\Sigma^{\wt \jmath} $ the section of $\Sigma$ at fixed $ \wt \jmath$, namely
$ \Sigma^{\wt \jmath} : = \{ (\ell,\wt \jmath, \fa) \in \Sigma \} $.
Given $K > 1$,  denote by $\Sigma_K$ any subset of singular sites of $D(\lambda, \theta)$ such that the cardinality of the section $\Sigma^{\wt \jmath}$ satisfies $\sharp \Sigma^{\wt \jmath}_K \leq K$,  for any $\wt \jmath \in \Sigma$ (see Definition \ref{def:sigmaKn}).

The key  result is the following bound on the length of $\Gamma$-chains of singular sites.
\begin{proposition}[Separation of singular sites for NLS]
\label{thm:lcn}
There exists a constant $C(\sL, d, n)>0$ such that for all $\lambda \in \Lambda$, $\theta \in \R$ and for all $K>1, \Gamma \geq 2$, any $\Gamma$-chain of singular sites of  $D(\lambda, \theta)$
in $\Sigma_K$  has length $L \leq (K \Gamma )^{C(\sL,d,n)}$.
\end{proposition}
We postpone the proof to the next section \ref{sec41} and we first prove Theorem \ref{main2} for  NLS.

\paragraph{Verification of the assumptions of Berti-Corsi-Procesi theorem.}

First,  \eqref{D.est} is trivially satisfied with $\sigma = 2$. Moreover,   provided $f \in C^\tq$ for $\tq$ large enough and $s_0 > (d+n)/2$,  the tame estimates (f1)--(f2) hold true.

We verify now Hypothesis 1--3 {concerning 
the operator $ D(\lambda )+ \epsilon T(u) $,
 $$
 T(u) = \begin{pmatrix}
 p(\vf, x) & q(\vf, x) \\
 \bar{q(\vf, x)} & p(\vf, x)
 \end{pmatrix} , \qquad 
 \begin{array}{l}
 p(\vf, x):= - \partial_{u^+} \tf^+(\vf, x, u^+(\vf,x), u^-(\vf, x)) \\
 q(\vf, x) := - \partial_{u^-} \tf^-(\vf, x, u^+(\vf,x),u^-(\vf, x)) 
 \end{array} \, , 
 $$
obtained linearizing \eqref{NLSavector}
at a point  $(u^+, u^-) \in  \cU$.
By the Hamiltonian structure \eqref{tfham}, the conditions on $\tf^\pm$ and the request $(u^+, u^-) \in \cU$, 
{the function $p(\vf, x) $ is real  and $q(\vf, x) $ is complex valued}.
\\[1mm]
  \underline{Verification of Hypothesis 1.}
 The covariance property \eqref{2.24a} holds with
  $\fD_{j,\fa}(y) = - \fa y + \mu_j + m$.
The {multiplication} operator $T(u)$,  is represented, 
in the exponential basis \eqref{basisNLS},  by the  {T\"opliz} matrix
{$$
T_{(\ell,j,\fa)}^{(\ell', j', \fa)} = p(j-j', \ell - \ell') \mbox{ if } \fa = \pm 1 \, , 
\  
T_{(\ell,j, 1)}^{(\ell', j', -1)} = q(j-j', \ell - \ell') \, , 
\ 
T_{(\ell,j, -1)}^{(\ell', j', 1)} = \bar{ q(j-j', \ell - \ell')} \, ,
$$
 where $ p(j, \ell)$, $q(j, \ell)$ are the Fourier coefficients 
of the functions $p(\vf, x), q(\vf, x)$}. 
 Then  \eqref{2.24b} holds and \eqref{2.24c}-\eqref{2.24d} hold  with $\sigma_0 = 0$.
 \\[1mm]
 \underline{Verification of Hypothesis 2.}
 By a direct computation Hypothesis 2 is met with $\mathfrak{n} = 2$. 
\\[1mm]
 \underline{Verification of Hypothesis 3.} It is the content of Proposition \ref{thm:lcn}.

\smallskip

The verification of the measure estimates \eqref{meas.bad1}  follows exactly as in \cite{bebo13}, and we omit it.
{Applying Theorem \ref{BCP1} we prove Theorem \ref{main2} for  NLS}.

\subsection{Proof of Proposition \ref{thm:lcn}}\label{sec41}

Let $(\ell_q, j_q, \fa_q)_{q=0,\ldots, L}$ be a $\Gamma$-chain
 of singular sites for $D(\lambda,\theta)$; in particular
 \begin{equation}
 \label{ffd}
 \max_{q = 0, \ldots, L-1} \left\lbrace |\ell_{q+1} - \ell_q| ,  \ 
 |j_{q+1} - j_q| \right\rbrace \leq \Gamma \, . 
 \end{equation}
As in Section \ref{sec:L}, we introduce the quadratic form and the associated scalar product 
\begin{equation*}
 \nla{y} := \norm{W y}^2   \, , \qquad 
 \pla{y}{ y'} = \la W y,  W y'  \ra_{\R^d} \, , \qquad y, y' \in \R^d \, .
 \end{equation*}  
We have the following lemma. 
\begin{lemma}
\label{lem:bound_s}
For all $ q_0, q \in \{0,\ldots, L\}$ we have
\begin{equation}
\label{bound_nls}
\abs{\pla{j_{q_0}}{ j_q - j_{q_0}} } \leq C(\sL,  d, n) |q-q_0|^2 \Gamma^2 \, .
\end{equation}
\end{lemma}

\begin{proof}
By the definition of singular sites, for all $  q \in \{0,\ldots, L\}$,  
$$
\begin{cases}
\abs{ \mu_{j_q} - \omega \cdot \ell_q   + m - \theta } < 1 & \mbox{ if } \fa_q = + 1 \, , \\
\abs{ \mu_{j_q} +\omega \cdot \ell_q   + m +\theta} < 1 & \mbox{ if } \fa_q = -1 \, ,
\end{cases}
$$
which leads to  one of the following $ \theta$-independent inequalities
$$
\abs{\pm \omega \cdot(\ell_{q+1} - \ell_q) + \mu_{j_{q+1}} \pm \mu_{j_q} }
\leq 2 (|m|+1) \, . 
$$
By \eqref{ffd} we deduce
$$
\abs{\mu_{j_{q+1}} \pm \mu_{j_q}} \leq 2 (|m|+1) + n|\omega| \Gamma \leq C_1 \Gamma 
$$
for some $C_1 := C_1(m, n)$.
Since $\abs{\mu_{j_{q+1}} - \mu_{j_q}} \leq \abs{\mu_{j_{q+1}} + \mu_{j_q}}$, in any case we obtain the bound
$$
\abs{\mu_{j_{q+1}} - \mu_{j_q}} \leq C_1 \Gamma \, , \quad \forall q \in \{0,\ldots, L\} \, .
$$
Therefore, for any $q, q_0 \in \{0,\ldots,  L\}$ we get 
$$
\abs{\nla{j_q} -  \nla{ j_{q_0}} }  = \abs{\mu_{j_{q}} - \mu_{j_{q_0}}} \leq C_1 |q-q_0| \, \Gamma \, . 
$$
Since
$\nla{j_q} = \nla{ j_{q_0}}  + 2 \pla{j_{q_0}}{ j_q - j_{q_0}} + 
\nla{ j_q - j_{q_0}}$, we deduce
\begin{align*}
\abs{\pla{j_{q_0}}{j_q - j_{q_0}}} \leq  \big| \nla{j_q} -  \nla{ j_{q_0}} \big|  + 
\nla{j_q-j_{q_0}} 
\leq C(\sL,  d, n) \Gamma^2 |q-q_0|^2  
\end{align*}
which is \eqref{bound_nls}. 
\end{proof}

At this point the proof closely follows that in Section \ref{sec:L} and we shall be short. 
Consider the subspace 
$$
\begin{aligned}
G &:= {\rm span }_\R \la j_q - j_{q'} \, , \  0 \leq q, q' \leq L \ra 
=   {\rm span }_\R \la j_q - j_{q_0} \, , \  q = 0, \ldots, L \ra \, , \ \  \forall q_0 = 0, \ldots, L \, , 
\end{aligned}
$$
with dimension $ g := \dim G$, $ 1 \leq g \leq d$. Decompose
$\R^d = G \oplus G^{\bot \pla{\cdot}{\cdot}} $ as in \eqref{GGbot}, 
and let   $P_G $ 
denote the orthogonal projector  on 
$G$. 
As in the proof of Proposition \ref{prop:lG}, in order  to bound the norm of each 
 $P_G j_{q_0} $ we consider two cases.
\\[1mm]
{\em Case 1.} For any  $ q_0 \in \{0, \ldots, L\}$, it results that  
$ {\rm span }_\R \la j_q - j_{q_0} \, ,  \ \ \abs{q-q_0} \leq L^\varsigma \ra = G $ with $  \varsigma  := [3 (d+2)d]^{-1} $.
Then we select a basis $f_1, \ldots, f_g$ of $ G $ 
extracted from $j_q-j_{q_0}$, $ \abs{q-q_0} \leq L^\varsigma $.
Proceeding as in the proof of Proposition 
\ref{prop:lG} and using estimate \eqref{bound_nls} in place of \eqref{G.prop}, one proves the bound
$$
 \sharp \left\lbrace  \, j_q : 0 \leq q \leq L \right\rbrace
\leq C(\sL,d,n) (L^\varsigma \Gamma)^{(d+2)d} . 
$$
By the assumptions of Proposition \ref{thm:lcn}, the $\Gamma$-chain belongs to $\Sigma_K$, thus the  cardinality of the set of singular sites in the $ \Gamma$-chain with fixed $ j_{q_0} $ is bounded by $ K $, namely
$$
\sharp{ \{ (\ell_q, j_q, \fa_q)_{q = 0, \ldots ,L} \colon  \ j_q = j_{q_0} \} } \leq K \, ,
$$
and therefore we get that 
$$
L \leq C(\sL,d,n) (L^\varsigma \Gamma)^{(d+2)d} K \, .
$$
Since $ \varsigma  (d+2)d \leq 1 / 2 $ we finally deduce that $L \leq (K \Gamma)^{c_1(\sL,d, n)}$.
 \\[1mm]
{\em Case 2.}  $\exists q_0 \in \{0, \ldots, L \} $ for which
$g_1 := \dim {\rm span }_\R \la j_q - j_{q_0} \, ,  \ \ \abs{q-q_0} \leq L^\varsigma \ra \leq g-1 .
$
Then one argues as in Case 2 of Proposition \ref{prop:lG}, 
obtaining a bound of the form $L \leq (K \Gamma)^{C(\sL,d, n)}$ as claimed,
completing  the proof of Proposition \ref{thm:lcn}.

\appendix

\section{Berti-Corsi-Procesi abstract Nash-Moser theorem }
\label{app:BCP}
We state here  Berti-Corsi-Procesi abstract theorem \cite{BCP} specified to the case when the index of  the space-Fourier component runs all over  $\Z^d$ (which is the situation in case the spatial variable $x$ belongs to $\T^d$). Consider  a scale of Hilbert sequence spaces defined in the following way. Define first the  index set
\begin{equation}\label{frakK}
\fK := \Z^n\times \Z^d \times \fA \ni (\ell, j, \fa) = k
\end{equation}
where the set  $\fA=\{1\}$ (for NLW) or 
$\fA=\{1,-1\}$ (for NLS). 
If $\fA =\{1\} $ we  simply write $k=(\ell , j)$.   
For $k= (\ell, j, \fa)$, $k' = (\ell', j', \fa') \in \fK$ we set
$$
{\rm dist}(k, k') := 
\begin{cases}
1 & \mbox{ if } \ell = \ell', j = j', \fa \neq \fa' \\
\max\{ |\ell - \ell'|, |j - j'|\} & \mbox{otherwise} \, 
\end{cases} ;
$$
here $|\ell|  := \max \{|\ell_1|, \cdots, |\ell_n|\}$,  $|j| := \max \{|j_1|, \cdots, |j_d|\}$.
 For any  $ s \geq 0 $, define the Sobolev space 
\begin{equation}\label{spazi}
H^{s} := H^{s}({\mathfrak K} ):=\Big\{u=\!\!
\{u_{k}\}_{k\in {\mathfrak K}} \, , \ u_k  \in \C   \,:\, 
\|u\|^2_{s}:= 
\sum_{k\in{\mathfrak K}} \langle w_k\rangle ^{2s}
|u_{k}|^{2}<\infty \Big\}
\end{equation}
where 
the weights $\langle w_k \rangle := \max(1, |\ell|, |j|  )$.

A bounded linear operator $M :H^s \to H^s$ is represented by an infinite dimensional matrix $(M_{k}^{k'})_{k,k' \in \fK}$.
We define now a norm on such operators.
\begin{definition}[$s$-decay norm] We say that a linear operator $M$ has finite $s$-decay norm if
$$
\abs{M}_s^2 := \sum_{(\ell, j) \in \Z^n \times \Z^d}  \max( 1, |\ell|, |j|  )^{2s}
\sup_{\ell_1 - \ell_2 = \ell \atop j_1 - j_2 = j}
\norm{M_{\{\ell_1,j_1\}}^{\{\ell_2, j_2\}}}_0 
 < \infty 
$$
where 
$ M_{\{\ell_1,j_1\}}^{\{\ell_2, j_2\}} := \left\lbrace M_{(\ell_1,j_1, \fa_1)}^{(\ell_2, j_2, \fa_2)}\right\rbrace_{\fa, \fa' \in \fA}  $
and $\norm{ \cdot}_0$ is the operatorial norm.
\end{definition}
Consider a nonlinear operator $F(\epsilon, \lambda, \cdot)\colon H^{s+\sigma} \to H^s$ of the form
\begin{equation}
\label{aop}
F(\epsilon, \lambda, u) = D(\lambda) u + \epsilon f(u)
\end{equation}
where $\epsilon>0$ is a small parameter, $\lambda \in \Lambda \subset [1/2, 3/2]$ and $D(\lambda)$ is a linear  diagonal operator $D(\lambda) : H^{s+\sigma} \to H^s$ fulfilling  
\begin{equation}
\label{D.est}
\norm{D(\lambda)h}_s , \  \norm{\partial_\lambda D(\lambda) h }_s  \leq C(s)  \norm{h}_{s + \sigma} . 
\end{equation}
The nonlinearity $f$ is assumed to be at least of class  $ C^2(B_1^{s_0}, H^{s_0})$ for some  $s_0>(d+ n)/2$, where $B_1^{s_0}$ is the ball of center zero and radius 1 in $H^{s_0}$.
We assume that the following tame estimates hold: given $ S' > s_0 $, for all $s \in[s_0, S')$ there exists a constant $C(s)>0$  such that for any 
$\|u\|_{s_0} \le2 $, 
\begin{itemize} 
\item[(f1)] $ \| \di f(u)[h]\|_{s}\le C(s) \big( \|u\|_{s} \|h\|_{s_0} +  \|h\|_{s} \big) $,

\item[(f2)] $\| \di^{2} f(u)[h,v]\|_{s} \le C(s)\Big(
\|u\|_{s}\|h\|_{s_0}\|v\|_{s_0}+\|h\|_{s}\|v\|_{s_0}
+\|h\|_{s_0}\|v\|_{s}\Big)$.
\end{itemize}
The main assumptions are on the operator which is obtained by linearizing \eqref{aop} at a point $u \in H^s$: we denote it by
\begin{equation}
\label{L.def}
L = L(\lambda, \epsilon, u) := D(\lambda) + \epsilon T(u)
\end{equation}
where $T(u) := \di f(u)$. 
\\[1mm]
\underline{Hypothesis 1.} Let $\bar \omega \in \R^d$ satisfy \eqref{omega1}. There exist a function
$\fD \colon \Z^d \times \fA \times \R \to \C$ and $\sigma_0 >0$ such that for all $ \|u\|_{s_0}, \|u'\|_{s_0}\le 2$ and $ s_0+\sigma_0 < s < S' $, one has 
\begin{align}
\label{2.24a}
&\mbox{(Covariance)} \qquad  D_{(\ell, j,  \fa)}(\lambda) = \fD_{(j, \fa)}(\lambda \bar\omega \cdot \ell), \qquad \forall \lambda \in \Lambda \, , \\
\label{2.24b}
&\mbox{(T\"oplitz in time)} \qquad T_{(\ell, j,  \fa)}^{(\ell', j', \fa')} = T_{(j, a)}^{(j', \fa')}(\ell - \ell')\, , \\
\label{2.24c}
&\mbox{(Off diagonal decay)} \qquad \vert T(u)\vert_{s-\sigma_0}\le C(s)(1+ \|u\|_{s}) \, ,   \\
\label{2.24d}
&\mbox{(Lipschitz)} \qquad  \vert T(u)-T(u')\vert_{s-\sigma_0} \le C(s)
\big( \|u-u'\|_s +(\|u\|_s + \|u'\|_s)\|u-u'\|_{s_0} \big)\, .
\end{align}
In order to state the next hypothesis, we define for any $\theta\in\R$ the infinite matrices
\begin{align*}
& D(\lambda,\theta) := {\rm Diag}(D_{(\ell, j,  \fa)}(\lambda,\theta)) \, , 
\quad
 D_{(\ell, j ,\fa)}(\lambda,\theta):= 
 \fD_{(j,\fa)}(\bar \omega \cdot \ell +\theta) \, , \\
& L(\epsilon,\lambda,\theta,u):= D(\lambda,\theta)+\epsilon T(u)\, .  
\end{align*}
\noindent
\underline{Hypothesis 2.}
There is $ \mathfrak n \in \N $ such that for all $\tau_1 > 1$, $N > 1 $, $\lambda\in \Lambda$, $(\ell, j, \fa) \in \fK$
the set
\begin{equation}\label{meas.diag}
\big\{\theta \in \R\;:\; |D_{(\ell,j,\fa)}(\lambda,\theta)|\le N^{-\tau_{1}}  \big\} 
\subseteq \bigcup_{q=1}^{\mathfrak n} I_{q} \quad \mbox{intervals with }  {\rm meas }(I_q) \le  N^{-\tau_1} \, . 
\end{equation}
The last hypothesis that is needed  concerns separation properties of  clusters of singular sites. To state it, we need some preliminary definitions.

\begin{definition}[$\Gamma$-chain] 
\label{gammachainabs}
Given $\Gamma \geq 2$, a sequence 
of distinct integer vectors  $(k_q)_{q=0,\ldots,  L} \subset\fK$  is  a 
{\em $\Gamma$-chain} if
$$
{\rm dist}(k_{q+1}, k_q) \leq \Gamma \, , \qquad \forall 0\leq q \leq L-1 \, . 
$$ 
The number $L$ is called the {\em length} of the chain.
\end{definition}

\begin{definition}[Singular sites] 
\label{def:ssn}
We say that $k \in \fK$ is a singular site  for the matrix $D:= {\rm Diag}(D_{k})$ if
$\abs{D_{k}} < 1$.
\end{definition}
For any $\Sigma \subseteq \fK$ and $\wt \jmath \in \Z^d $, we denote
by $\Sigma^{\wt \jmath} $ the section of $\Sigma$ at fixed $ \wt \jmath$, namely
$$
\Sigma^{\wt \jmath} : = \{ k=( \ell,\wt \jmath, \fa) \in \Sigma \} \, . 
$$
\begin{definition}
\label{def:sigmaKn}
Let $\theta, \lambda$ be fixed and $K > 1$. We denote by $\Sigma_K$ any subset of singular sites of $D(\lambda, \theta)$ such that the cardinality of the sections $\Sigma^{\wt \jmath}$ satisfies $\sharp \Sigma^{\wt \jmath}_K \leq K$,  for any $\wt \jmath \in \Sigma$.
\end{definition}

\noindent \underline{Hypothesis 3.}
There exist a constant $C(\sL, d,n)>0$ and for any $N_0 \geq 2$, a set $\wt\Lambda	 = \wt\Lambda(N_0)$ such that for all $\lambda \in \wt\Lambda$, $\theta \in \R$ and for all $K>1, \Gamma \geq 2$ with $\Gamma K \geq N_0$, any $\Gamma$-chain of singular sites  of $D(\lambda, \theta)$ in $\Sigma_K$ (as in Definition \ref{def:sigmaKn}) has length $L \leq (K \Gamma )^{C(\sL, d,n)}$.
\\[1mm]
Given a family of matrices $L(\theta)$ parametrized by a parameter $\theta\in\R$ and $N>1$, 
for any $k=(\ell,j,\fa)\in\fK$ we denote by $L_{N,\ell,j}(\theta)$  the 
sub-matrix of $L(\theta)$ centered at $ (\ell,j) $, i.e.
$$
L_{N,\ell,j}(\theta):= \big\{ L_{k}^{k'}(\theta): \; {\rm dist}(k,k')\leq N \big\} \, .
$$
 For $\tau >0$, $N_0 \geq 1$, set
\begin{equation}
\label{barI}
\bar\Lambda := \bar\Lambda(N_0, \tau) := \left\lbrace
\lambda \in \Lambda \colon \abs{D_k(\lambda)} \geq N_0^{-\tau} \, , \  \forall k = (\ell, j, \fa) \in \fK \colon
\max\{ |\ell|, |j| \} \leq N_0 \right\rbrace \, . 
\end{equation}

\begin{theorem}[Berti-Corsi-Procesi]
\label{BCP1}
Let $ \fe > d + n + 1  $. 
Assume that  $ F $ in \eqref{aop}  satisfies \eqref{D.est}, (f1)--(f2) and
Hypotheses 1--3
with $ S' $ large enough, depending on $ \fe $.
Then, there are  $\tau_1>1$, $\bar{ N}_0 \in \N $,  
 $ s_1 $, $ S \in (s_0 + \sigma_0, S'-\sigma_0) $  with $s_1<S$ (all depending on $\fe$)  and  
 $ c(S) >  0 $ such that for all $N_0\ge \bar{N}_0 $, if the smallness condition 
\begin{equation}\label{epN0}
 \epsilon  N_0^{S} < c(S)
 \end{equation} 
holds, then the following holds: 
\begin{enumerate}
\item
{\bf (Existence)} 
There exist  
 a function $ u_\epsilon \in C^{1}(\Lambda, H^{s_{1}+\sigma})$  with $u_0(\lambda)=0$, 
and 
a set $  \cC_{\epsilon} \subset \Lambda  $ (defined  only in terms of $ u_\epsilon $), 
such that,  for all $ \lambda\in \cC_{\epsilon}   $, we have 
$ F(\epsilon,\lambda,u_\epsilon (\lambda))=0 $. 
\item
{\bf (Measure estimate)}
Let $ N_0 = [\e^{-1/(S+1)}] $ with $ \e $ small enough so that \eqref{epN0} holds. 
Assume 
that for all $N\geq N_0 $, 
\begin{equation}\label{meas.bad1}
{\rm meas}(\Lambda \setminus {\bar \cG}_{N}^{0} ), \  {\rm meas}(\Lambda \setminus {\bar \fG}_{N} ) 
= O( N^{-1} )  \, , \quad 
 {\rm meas}(\Lambda \setminus (\bar \Lambda\cap\wt\Lambda)) = O (N_0^{- 1 }) \, ,
\end{equation}
where $ \wt\Lambda=\wt\Lambda(N_0)$
 is defined in Hypothesis 3, $ \bar \Lambda  $
in \eqref{barI}, 
and, 
for all $ N \in \N $, 
$$
\bar \fG_{N}:= \Big\{\lambda \in \Lambda \;:\; \|L^{-1}_{N,0,0}(\epsilon,\lambda,u_\epsilon (\lambda))\|_{0}\leq
 N^{\tau_1}/2 \Big\} \, , 
$$
$$
 \begin{aligned}
 \bar \cG^0_{N} & :=\Big\{ \lambda \in\Lambda\;:\;
 \forall  j_{0} \in \Z^d  \mbox{ there is a covering } \\
& \qquad  \bar B^0_{N} (j_{0},\e,\lambda)\subset \bigcup_{q=1}^{N^{\fe}}I_{q},
 \mbox{ with }I_{q}=I_{q}(j_{0})\mbox{ intervals with }
 {\rm meas}(I_{q})\le  N^{-\tau_1}
 \Big\}
  \end{aligned}
$$
 where
 $$
\bar B^0_{N}(j_{0},\e,\lambda):= \Big\{\theta\in\R\;:\;
\|L_{N,0,j_{0}}^{-1}(\e,\lambda,\theta, u_\e(\lambda))\|_0>N^{\tau_1}/2\Big\} \, . 
$$
 Then 
 $\mathcal C_\e$ satisfies, for some $ \mathtt K > 0 $,   the measure estimate 
$ {\meas}(\Lambda \setminus \mathcal C_\e ) \le \mathtt K \e^{1/(S+1)} $. 
\end{enumerate}
 \end{theorem}

\small

\def\cprime{$'$}


\begin{thebibliography}{DLvStvc08}

\bibitem[BGMR17]{BGMR2}
D.~{Bambusi}, B.~{Gr{\'e}bert}, A.~{Maspero}, and D.~{Robert}.
\newblock {Growth of Sobolev norms for abstract linear Schr{\"o}dinger
  Equations}.
\newblock {\em J. Eur. Math. Soc. (JEMS)}, in press, 2017.

\bibitem[BGMR18]{BGMR1}
D.~Bambusi, B.~Gr{\'e}bert, A.~Maspero, and D.~Robert.
\newblock Reducibility of the quantum harmonic oscillator in d-dimensions with
  polynomial time-dependent perturbation.
\newblock {\em Anal. PDE}, 11(3):775--799, 2018.



\bibitem[BB12]{bebo12}
M.~Berti and P.~Bolle.
\newblock Sobolev quasi-periodic solutions of multidimensional wave equations
  with a multiplicative potential.
\newblock {\em Nonlinearity}, 25(9):2579--2613, 2012.

\bibitem[BB13]{bebo13}
M.~Berti and P.~Bolle.
\newblock Quasi-periodic solutions with {S}obolev regularity of {NLS} on 
{${\mathbb  T}^d$} with a multiplicative potential.
\newblock {\em J. Eur. Math. Soc. (JEMS)}, 15(1):229--286, 2013.

\bibitem[BBP10]{BBP}
M.~Berti, P.~Bolle, and M.~Procesi.
\newblock An abstract {N}ash-{M}oser theorem with parameters and applications
  to {PDE}s.
\newblock {\em Ann. Inst. H. Poincar\'{e} Anal. Non Lin\'{e}aire},
  27(1):377--399, 2010.

\bibitem[BCP15]{BCP}
M.~Berti, L.~Corsi, and M.~Procesi.
\newblock An abstract {N}ash-{M}oser theorem and quasi-periodic solutions for
  {NLW} and {NLS} on compact {L}ie groups and homogeneous manifolds.
\newblock {\em Comm. Math. Phys.}, 334(3):1413--1454, 2015.



\bibitem[BP11]{BP}
M.~Berti and M.~Procesi.
\newblock Nonlinear wave and {S}chr\"{o}dinger equations on compact {L}ie
  groups and homogeneous spaces.
\newblock {\em Duke Math. J.}, 159(3):479--538, 2011.


\bibitem[BK18]{BK}
P.~Bolle and C. Khayamian.
\newblock Quasi periodic solutions for the forced NLW equation with potential
on Zoll manifolds.
\newblock {\em preprint} 2018.



\bibitem[Bou98]{bou98}
J.~Bourgain.
\newblock Quasi-periodic solutions of {H}amiltonian perturbations of 2{D}
  linear {S}h\"odinger equation.
\newblock {\em Ann. Math.}, 148:363--439, 1998.

\bibitem[Bou99a]{bou99}
J.~Bourgain.
\newblock Growth of {S}obolev norms in linear {S}chr\"odinger equations with
  quasi-periodic potential.
\newblock {\em Comm. Math. Phys.}, 204(1):207--247, 1999.

\bibitem[Bou99b]{bourgain99}
J.~Bourgain.
\newblock On growth of {S}obolev norms in linear {S}chr\"odinger equations with
  smooth time dependent potential.
\newblock {\em J. Anal. Math.}, 77:315--348, 1999.

\bibitem[Bou05]{Bou05}
J.~Bourgain.
\newblock { Green's function estimates for lattice {S}chr\"{o}dinger
  operators and applications}, volume 158 of {\em Annals of Mathematics
  Studies}.
\newblock Princeton Univ. Press, Princeton, NJ, 2005.


\bibitem[Bou07]{bou07}
J.~Bourgain.
\newblock On {S}trichartz's inequalities and the nonlinear {S}chr\"{o}dinger
  equation on irrational tori.
\newblock In {\em Mathematical aspects of nonlinear dispersive equations},
  volume 163 of {\em Ann. of Math. Stud.}, pages 1--20. Princeton Univ. Press,
  Princeton, NJ, 2007.

\bibitem[BD15]{bourgain15}
J.~Bourgain and C.~Demeter.
\newblock The proof of the $l^2$ decoupling conjecture.
\newblock {\em Ann. Math.}, 182(1):351--389, 2015.


\bibitem[CW10]{catoire10}
F.~Catoire and W.-M. Wang.
\newblock Bounds on {S}obolev norms for the defocusing nonlinear
  {S}chr\"{o}dinger equation on general flat tori.
\newblock {\em Comm. Pure Appl. Anal.}, 9(2):483--491, 2010.

\bibitem[CHP15]{CHP}
L.~Corsi, E.~Haus, and M.~Procesi.
\newblock A {KAM} result on compact {L}ie groups.
\newblock {\em Acta Appl. Math.}, 137:41--59, 2015.



\bibitem[Del10]{del2}
J.-M. Delort.
\newblock Growth of {S}obolev norms of solutions of linear {S}chr\"odinger
  equations on some compact manifolds.
\newblock {\em Int. Math. Res. Not. (IMRN)}, (12):2305--2328, 2010.

\bibitem[Del14]{del}
J.-M. Delort.
\newblock Growth of {S}obolev norms for solutions of time dependent
  {S}chr\"odinger operators with harmonic oscillator potential.
\newblock {\em Comm. PDE}, 39(1):1--33, 2014.

\bibitem[{Den}17]{deng217}
Y.~{Deng}.
\newblock {On growth of Sobolev norms for energy critical NLS on irrational tori: small energy case}.
\newblock {\em Comm. Pure App. Math.}, DOI:10.1002/cpa.21797,  2018.

\bibitem[DG17]{deng-germain17}
Y.~Deng and P.~Germain.
\newblock Growth of solutions to NLS on irrational tori.
\newblock {\em Int. Math. Res. Not. (IMRN)}, 2017.

\bibitem[DGG17]{deng17}
Y.~Deng, P.~Germain, and L.~Guth.
\newblock Strichartz estimates for the {S}chr\"{o}dinger equation on irrational
  tori.
\newblock {\em J. Funct. Anal.}, 273(9):2846--2869, 2017.

\bibitem[DLS08]{duclos}
P.~Duclos, O.~Lev, and P. Stov\'icek.
\newblock On the energy growth of some periodically driven quantum systems with
  shrinking gaps in the spectrum.
\newblock {\em J. Stat. Phys.}, 130(1):169--193, 2008.

\bibitem[EGK16]{EGK}
H.~Eliasson, B.~Gr\'{e}bert, and S.~Kuksin.
\newblock K{AM} for the nonlinear beam equation.
\newblock {\em Geom. Funct. Anal. (GAFA)}, 26(6):1588--1715, 2016.

\bibitem[EK09]{EK1} H. Eliasson and S. Kuksin, 
On reducibility of Schr\"odinger equations with
quasiperiodic in time potentials, {\it Comm. Math. Phys.}, 286, 125-135, 2009.

\bibitem[EK10]{EK10}
H.~Eliasson and S.~Kuksin.
\newblock K{AM} for the nonlinear {S}chr\"{o}dinger equation.
\newblock {\em Ann. of Math. (2)}, 172(1):371--435, 2010.


\bibitem[FZ12]{fang}
D.~Fang and Q.~Zhang.
\newblock On growth of {S}obolev norms in linear {S}chr\"odinger equations with
  time dependent {G}evrey potential.
\newblock {\em J. Dynam. Diff.  Eq.}, 24(2):151--180, 2012.



\bibitem[Gan59]{gantmacher}
F. Gantmacher.
\newblock {\em The Theory of Matrices. vol. 1 and vol. 2. }, Chelsea Publishing Company, New York 68, 1959. 

\bibitem[GXY11]{geng11}
J.~Geng, X.~Xu, and J.~You.
\newblock An infinite dimensional {KAM} theorem and its application to the two
  dimensional cubic {S}chr\"{o}dinger equation.
\newblock {\em Adv. Math.}, 226(6):5361--5402, 2011.



\bibitem[GP16]{GP2}
B.~{Gr{\'e}bert} and E.~Paturel.
\newblock {KAM for the Klein Gordon equation on $\mathbb S^d$}.
\newblock {\em Boll. Unione Mat. Ital.}, 9(2):237--288, 2016.






\bibitem[GHHMP18]{GHHMP}
M. Guardia, Z. Hani, E. Haus, A. Maspero, M. Procesi.
\newblock Strong nonlinear instability and growth of Sobolev norms near quasiperiodic finite-gap tori for the 2D cubic NLS equation
\newblock {\em ArXiv e-prints}, arXiv:1810.03694,
  2018.




\bibitem[GOW14]{guo-oh}
Z.~Guo, T.~Oh, and Y.~Wang.
\newblock Strichartz estimates for {S}chr\"{o}dinger equations on irrational
  tori.
\newblock {\em Proc. Lond. Math. Soc. (3)}, 109(4):975--1013, 2014.



\bibitem[{Mas}18]{ma18}
A.~{Maspero}.
\newblock {Lower bounds on the growth of Sobolev norms in some linear time
  dependent Schr\"odinger equations}.
\newblock {\em Math. Res. Lett. (MRL)}, in press,  2018.

\bibitem[MP18]{maspero_procesi}
A.~Maspero and M.~Procesi.
\newblock Long time stability of small finite gap solutions of the cubic
  nonlinear Schr\"{o}dinger equation on ${\mathbb T}^2$.
\newblock {\em J. Diff.  Eq.}, 265(7):3212--3309, 2018.

\bibitem[MR17]{maro}
A.~Maspero and D.~Robert.
\newblock On time dependent {S}chr{\"o}dinger equations: {G}lobal
  well-posedness and growth of {S}obolev norms.
\newblock {\em J.  Funct.  Anal.}, 273(2):721 -- 781, 2017.

\bibitem[Nen97]{nen}
G.~Nenciu.
\newblock Adiabatic theory: stability of systems with increasing gaps.
\newblock {\em Annales de l'I. H. P.}, 67-4:411--424, 1997.

\bibitem[PTV17]{visciglia17}
F.~Planchon, N.~Tzvetkov, and N.~Visciglia.
\newblock On the growth of {S}obolev norms for {NLS} on 2- and 3-dimensional
  manifolds.
\newblock {\em Anal. PDE}, 10(5):1123--1147, 2017.

\bibitem[PP15]{PP15}
C.~Procesi and M.~Procesi.
\newblock A {KAM} algorithm for the resonant non-linear {S}chr\"{o}dinger
  equation.
\newblock {\em Adv. Math.}, 272:399--470, 2015.

\bibitem[PP16]{pp_bumi}
M.~Procesi and C.~Procesi.
\newblock Reducible quasi-periodic solutions for the non linear {S}chr\"odinger
  equation.
\newblock {\em Boll. Unione Mat. Ital.}, 9(2):189--236, 2016.



\bibitem[SW18]{staffilani_wilson18}
G.~{Staffilani} and B.~{Wilson}.
\newblock {Stability of the Cubic Nonlinear Schrodinger Equation on Irrational
  Tori}.
\newblock {\em ArXiv e-prints}, 	arXiv:1806.01635,  2018.


\bibitem[Wan16]{WDuke}
W.-M. Wang.
\newblock Energy supercritical nonlinear {S}chr\"{o}dinger equations:
  quasiperiodic solutions.
\newblock {\em Duke Math. J.}, 165(6):1129--1192, 2016.

\bibitem[Wan17]{Wpre}
W.-M. Wang. 
\newblock Quasi-Periodic Solutions For Nonlinear Klein-Gordon Equations.
\newblock {\em ArXiv e-prints}, arXiv:1609.00309,  2016.



\end{thebibliography}
\end{document}